\documentclass[journal]{IEEEtran}
\ifCLASSINFOpdf
\else
\fi
\usepackage{def}

\begin{document}
%
\title{NPGA: A Unified Algorithmic Framework for Decentralized Constraint-Coupled Optimization}
%
%
%
%
%
%
\author{Jingwang Li and Housheng Su
  \thanks{The authors are with the School of Artificial Intelligence and Automation and the Key Laboratory of Image Processing and Intelligent Control of Education Ministry of China, Huazhong University of Science and Technology, Wuhan 430074, China (email: jingwangli@outlook.com, houshengsu@gmail.com).}
  \thanks{Digital Object Identifier \href{https://doi.org/10.1109/TCNS.2024.3354882}{10.1109/TCNS.2024.3354882}}
}

\maketitle

\begin{abstract}
  This work focuses on a class of general decentralized constraint-coupled optimization problems. We propose a novel nested primal-dual gradient algorithm (NPGA), which can achieve linear convergence under the weakest known condition, and its theoretical convergence rate surpasses all known results. More importantly, NPGA serves not only as an algorithm but also as a unified algorithmic framework, encompassing various existing algorithms as special cases. By designing different network matrices, we can derive numerous versions of NPGA and analyze their convergences by leveraging the convergence results of NPGA conveniently, thereby enabling the design of more efficient algorithms. Finally, we conduct numerical experiments to compare the convergence rates of NPGA and existing algorithms, providing empirical evidence for the superior performance of NPGA.
\end{abstract}

\begin{IEEEkeywords}
  Constraint-coupled optimization, nested primal-dual gradient algorithm, linear convergence, unified algorithmic framework.
\end{IEEEkeywords}

%
\IEEEpeerreviewmaketitle

\section{Introduction}
Consider a networked system composing of $n$ agents, the network topology is represented by an undirected graph $\mathcal{G}=(\mathcal{N}, \mathcal{E})$, where $\mathcal{N} = \{1,\cdots, n\}$ and $\mathcal{E} \subset \mathcal{N} \times \mathcal{N}$ denote the sets of nodes and edges, respectively. Specifically, $(i, j) \in \mathcal{E}$ if agents $i$ and $j$ can communicate with each other. In this work, we consider the following decentralized constraint-coupled optimization problem:
\begin{equation} \label{pro2} \tag{P1}
  \begin{aligned}
    \min_{x_i \in \mR^{d_i}} \  & \sum_{i=1}^n \pare{f_i(x_i) + g_i(x_i)} + h\lt(\sum_{i=1}^{n}A_ix_i\rt),
  \end{aligned}
\end{equation}
where $f_i:\mR^{d_i} \rightarrow \mR$, $g_i:\mR^{d_i} \rightarrow \mR \cup \{+\infty\}$, and $A_i \in \mR^{p \times d_i}$ are completely private to agent $i$. This means that agent $i$ cannot share this information with other agents in the system. On the other hand, $h:\mR^p \rightarrow \mR \cup \{+\infty\}$ is public and known to all agents. Without loss of generality, we assume that the solution set of \cref{pro2} is nonempty.

Different from the classical decentralized unconstrained optimization problem
\begin{equation} \label{uc} \tag{P2}
  \min_{x \in \mR^m} \sum_{i=1}^{n}\phi_i(x),
\end{equation}
\cref{pro2} introduces an additional cost function $h$ that couples the decision variables of all agents. Notice that we can reformulate \cref{pro2} as a new problem with a globally coupled equality constraint $\sum_{i=1}^{n}A_ix_i - y = 0$ by introducing a relaxed variable $y$, which is the reason we refer to it as a decentralized constraint-coupled optimization problem \cite{falsone2020tracking}. Many real-world optimization problems can be modeled or reformulated as \cref{pro2}, including vertically distributed regression and vertical federated learning \cite{chang2014multi,liu2022vertical}, distributed basis pursuit \cite{mota2011distributed}, and distributed resource allocation \cite{yi2016initialization,zhu2019distributed,zhang2020distributed}.

In this work, we aim to answer two questions:
\begin{enumerate}[(i)]
  \item Can we develop linearly convergent distributed algorithms for \cref{pro2} with weaker convergence conditions and faster convergence rates?
  \item Can we design a unified algorithmic framework for \cref{pro2}?
\end{enumerate}

Regarding the first question, while there have been some linearly convergent distributed algorithms proposed for problems similar to \cref{pro2}, most of them can only handle specific cases of \cref{pro2}. This means that stronger conditions need to be imposed on \cref{pro2} to make it solvable by these algorithms. These conditions include having $g_i=0$ and $h$ being an indicator function corresponding to the equality constraint \cite{kia2017distributed,yi2016initialization,zhu2019distributed,zhang2020distributed,nedic2018improved,chang2014multi,alghunaim2019proximal,li2022implicit}, $A_i=1$ \cite{kia2017distributed}, $A_i=\mI$ \cite{yi2016initialization,zhu2019distributed,zhang2020distributed,nedic2018improved}, or $A_i$ has full row rank \cite{chang2014multi,alghunaim2019proximal}. However, these conditions are often too restrictive and cannot be satisfied by most real-world problems. For example, in the context of vertical federated learning, the dimensions of $A_i$ correspond to the numbers of samples and features that agent $i$ owns, respectively. However, the number of samples is usually much larger than that of features, which means that the conditions about $A_i$ mentioned above cannot be met. Therefore, there is a need for linearly convergent distributed algorithms with weaker conditions. Some recent progress has been made in this direction \cite{li2022implicit,alghunaim2021dual}. In \cite{li2022implicit}, an implicit tracking-based distributed augmented primal-dual gradient dynamics (IDEA) is proposed, which can converge without the need for the strong convexity of $f_i$ and achieve linear convergence under the condition that $g_i=0$, $h$ is an indicator function corresponding to the equality constraint, and $A = [A_1, \cdots, A_n]$ has full row rank, which is much weaker than the above conditions about $A_i$. On the other hand, \cite{alghunaim2021dual} focuses solely on linear convergence, and the dual consensus proximal algorithm (DCPA) is proposed, which achieves linear convergence under the condition that $g_i=0$ and $A$ has full row rank or $h$ is smooth. However, there are some drawbacks to both IDEA and DCPA:
\begin{enumerate}[(i)]
  \item Though the linear convergence rate of IDEA has been established, as a continuous-time algorithm, it needs to be discretized for practical implementation. However, the upper bounds of the discretized step-sizes and the corresponding convergence rate of its discretization are not clear.
  \item For DCPA, the linear convergence rate is obtained by inevitably imposing a smaller bound on one of its step-sizes, which leads to a looser convergence rate.
  \item Many Adapt-then-Combine (ATC) algorithms have been proposed for \cref{uc} \cite{nedic2017geometrically,xu2015augmented,li2019decentralized}, and it has been shown that ATC algorithms have additional favorable properties compared to Combine-then-Adapt (CTA) ones \footnote{There are two types of information diffusion strategies over networks: CTA and ATC \cite{sayed2014diffusion}. Accordingly, existing distributed optimization algorithms can be categorized as CTA or ATC algorithms based on their information diffusion strategies.}, such as the ability to use uncoordinated step-sizes or step-sizes as large as their centralized counterparts \cite{nedic2017geometrically,li2019decentralized}. However, both IDEA and DCPA are CTA algorithms, thus it remains open problems whether ATC algorithms can linearly solve \cref{pro2} under the same or weaker conditions as CTA ones and if they are superior in other aspects.

\end{enumerate}

In terms of the second question, several unified algorithmic frameworks have been proposed for the decentralized unconstrained optimization problem \cref{uc} \cite{alghunaim2020decentralized,xu2021distributed,li2022gradient,wu2022unifying}. However, to the best of our knowledge, there is no such framework for \cref{pro2}. Given that there are only two algorithms capable of linearly solving \cref{pro2} under weaker conditions, it becomes crucial to develop a unified algorithmic framework that enables the convenient design of new algorithms and facilitates the analysis of their convergences.

The major contributions of this work are summarized as follows:
\begin{enumerate}[(i)]
  \item Regarding the second question, we propose a novel nested primal-dual gradient algorithm for \cref{pro2}, which is a unified algorithmic framework. NPGA involves the configuration of three network matrices, allowing it to encompass existing algorithms, such as DCPA and the dual coupled diffusion algorithm (DCDA) \cite{alghunaim2019proximal}, as special cases. To the best of our knowledge, this is the first unified algorithmic framework for solving \cref{pro2} and its variants. By designing new network matrices within NPGA, we can generate novel distributed algorithms, including both ATC and CTA algorithms, for \cref{pro2}. The convergence analysis of these new algorithms can be conveniently conducted by leveraging the convergence results of NPGA.
  \item For the first question, we prove that NPGA achieves linear convergence under the weakest known condition: $g_i=0$ and $A = [A_1, \cdots, A_n]$ has full row rank, or $h$ is smooth. As mentioned before, DCPA can achieve linear convergence under the same condition. However, DCPA is merely a special case of NPGA, and its convergence results can be fully encompassed by those of NPGA. Furthermore, we prove that certain ATC variants of NPGA exhibit larger upper bounds of step-sizes and tighter convergence rates compared to DCPA. Numerical experiments also demonstrate that certain versions of NPGA have much faster convergence rates than DCPA in terms of both the numbers of iterations and communication rounds, which correspond to computational and communication costs, respectively.
  \item For the special case of \cref{pro2} with $g_i=0$ and $h$ representing an indicator function corresponding to the equality constraint, we derive a weaker condition concerning the network matrices of NPGA to ensure linear convergence and obtain a tighter convergence rate. This extension significantly expands the range of network matrix choices within NPGA. As a special case of NPGA, the convergence rate of DCPA can also be improved. Additionally, we show that certain ATC versions of NPGA can employ step-sizes as large as the centralized version of NPGA, resulting in a larger stability region compared to other versions of NPGA, including DCPA.
\end{enumerate}

\section{Preliminaries} \label{preliminaries}
In this section, we introduce the notations and lemmas utilized throughout the paper.

\textit{Notations:} $\1_n$ and $0_n$ represent the vectors of $n$ ones and zeros, respectively. $\mI_n$ denotes the $n\times n$ identity matrix, and $0_{m \times n}$ denotes the $m\times n$ zero matrix. Noticed that if the dimensions can be inferred from the context, we would not explicitly indicate them. For a positive semi-definite matrix $B \in \mR^{n \times n}$, we define $\|x\|^2_B = x^\top Bx$. Given two symmetric matrices $A \in \mR^{n \times n}$ and $B \in \mR^{n \times n}$, $A > B$ ( $A \geq B$) indicates that $A-B$ is positive definite (positive semi-definite). For $A \in \mR^{n \times n}$, $\eta_1(A)$, $\underline{\eta}(A)$, and $\overline{\eta}(A)$ denote the smallest, smallest non-zero, and largest eigenvalues of $A$, respectively. For $B \in \mR^{m \times n}$, $\us(B)$, $\sigma_i(B)$, and $\os(B)$ represent the smallest non-zero, $i$-th smallest, and largest singular values of $B$, respectively. $\Null(B)$ and $\col(B)$ denote the null space and the column space of $B$, respectively. For a vector $v \in \mR^n$, $\Span(v)$ denotes its span.
For a function $f: \mR^m \rightarrow \mR$, $\mathcal{S}_f(x)$ denotes any subgradient of $f$ at $x$, $\partial f(x)$ denotes the subdifferential of $f$ at $x$, which is the set of all subgradients of $f$ at $x$, and $\prox_{\alpha f}(x) = \arg\min_{y} f(y) + \frac{1}{2\alpha}\norm{y-x}^2$ represents the proximal operator of $f$ with step-size $\alpha$.
Finally, $\text{diag}(\cdot)$ denotes the (block) diagonal matrix.

\begin{lemma}[\cite{nesterov2018lectures}] \label{convex_inequality}
  Let $f: \mR^{m} \rightarrow \mR$ be a  $\mu$-strongly convex and $l$-smooth function, it holds that
  \eqe{
    f(x) \geq f(y) + \dotprod{\nabla f(y), x-y} + \frac{\mu}{2}\norm{x-y}^2, \\
    f(x) \leq f(y) + \dotprod{\nabla f(y), x-y} + \frac{l}{2}\norm{x-y}^2, \\
    \dotprod{\nabla f(x) - \nabla f(y), x-y} \geq \mu\|x-y\|^2, \\
    \|\nabla f(x) - \nabla f(y)\|^2 \leq l\dotprod{\nabla f(x) - \nabla f(y), x-y},
    \nonumber
  }
  for any $x, y \in \mR^m$.
\end{lemma}

\begin{lemma} \label{min_max}
  Given two matrices $M \in \mR^{m \times n}$ and $H \in \mR^{n \times n}$, where $H$ is symmetric, it holds that $\sigma_i(MH) \leq \os(M)\sigma_i(H)$,  $i = 1, \cdots, n$.
\end{lemma}
\begin{proof}
  See Appendix \ref{appendix_min_max}.
\end{proof}
\begin{table*}[t]
  \caption{\cite{alghunaim2020decentralized} Existing distributed unconstrained optimization algorithms unified by PUDA. $L\in \mR^{n\times n}$ and $W\in \mR^{n\times n}$ are the Laplacian and mixing matrices associated with $\mathcal{G}$, respectively, and $c>0$ is a tunable constant. The last column denotes the number of communication rounds of the corresponding algorithm at each iteration.}
  \label{tab1}
  \centering
  \begin{tabular}{cccccc}
    \toprule
    \multicolumn{2}{c}{Existing algorithms} & $B^2$                                                 & $C$                  & $D$                  & Communication rounds     \\
    \midrule
    \multirow{3}{*}{CTA algorithms}         & DIGing\cite{qu2017harnessing,nedic2017achieving}      & $(\mI-W)^2$          & $\mI-W^2$            & $\mI$                & 2 \\
    ~                                       & EXTRA \cite{shi2015extra}                             & $\frac{1}{2}(\mI-W)$ & $\frac{1}{2}(\mI-W)$ & $\mI$                & 1 \\
    ~                                       & DLM \cite{ling2015dlm}                                & $c\alpha L$          & $c\alpha L$          & $\mI$                & 1 \\
    ~                                       & P2D2 \cite{alghunaim2019linearly}                     & $\frac{c}{2}(\mI-W)$ & $\frac{1}{2}(\mI-W)$ & $\mI$                & 1 \\
    \midrule
    \multirow{4}*{ATC algorithms}           & Aug-DGM \cite{xu2015augmented,nedic2017geometrically} & $(\mI-W)^2$          & $0$                  & $W^2$                & 2 \\
    ~                                       & ATC tracking \cite{di2016next,scutari2019distributed} & $(\mI-W)^2$          & $\mI-W$              & $W$                  & 2 \\
    ~                                       & Exact diffusion \cite{yuan2018exact}                  & $\frac{1}{2}(\mI-W)$ & $0$                  & $\frac{1}{2}(\mI+W)$ & 1 \\
    ~                                       & NIDS \cite{li2019decentralized}                       & $c(\mI-W)$           & $0$                  & $\mI-c(\mI-W)$       & 1 \\
    \bottomrule
  \end{tabular}
\end{table*}
\section{Algorithm Design} \label{design}
Let $d=\sum_{i=1}^n d_i$ and $\mx = [x_1\T, \cdots, x_n\T]\T \in \mR^d$, \cref{pro2} can be reformulated as
\begin{equation} \label{mP} \tag{P3}
  \begin{aligned}
    \min_{\mx \in \mR^d} \  & f(\mx) + g(\mx) + h(A\mx), \\
  \end{aligned}
\end{equation}
where $f(\mx) = \sum_{i=1}^n f_i(x_i)$, $g(\mx) = \sum_{i=1}^n g_i(x_i)$, and $A = [A_1, \cdots, A_n] \in \mR^{p \times d}$. Let $\mf(\mx) = f(\mx) + g(\mx)$, and introduce a slack variable $y \in \mR^p$, then we can rewrite \cref{mP} as
\begin{equation} \label{1613}
  \begin{aligned}
    \min_{\mx \in \mR^d, y \in \mR^p} \  & \mf(\mx) + h(y) \\
    \text{s.t.} \                        & A\mx - y = 0.
  \end{aligned}
\end{equation}
The corresponding saddle-point problem is given by
\eqe{
  &\min_{\mx \in \mR^d, y \in \mR^p} \max_{\lambda \in \mR^p} \mf(\mx) + h(y) + \dotprod{\lambda, A\mx - y} \\
  = &\max_{\lambda \in \mR^p}\lt(\min_{\mx \in \mR^d} \mf(\mx) + \lambda\T A\mx + \min_{y \in \mR^p} h(y) - \lambda\T y\rt) \\
  = &\max_{\lambda \in \mR^p}\min_{\mx \in \mR^d} \underbrace{\mf(\mx) + \lambda\T A\mx - h^*(\lambda)}_{\cL(\mx, \lambda)},
}
where $\lambda \in \mR^p$ is the Lagrange multiplier, $h^*(\lambda) = \sup_{y \in \mR^p} \lambda\T y - h(y)$ is the Fenchel conjugate of $h$, and $\cL$ is the Lagrangian function of \cref{mP}.
Throughout the work, we assume that the following assumptions hold.
\begin{assumption} \label{existence_saddle_point}
  There exists at least a saddle point to $\cL$.
\end{assumption}
Under \cref{existence_saddle_point}, \cref{mP} is equivalent to the saddle-point problem
\eqe{ \label{saddle_point_problem}
  \min_{\mx \in \mR^d}\max_{\lambda \in \mR^p} \ \mf(\mx) + \lambda\T A\mx - h^*(\lambda),
}
that is, if $(\mx^*, \lambda^*)$ solves \cref{saddle_point_problem}, then $\mx^*$ solves \cref{mP}. Consequently, we are able to solve \cref{mP} via solving \cref{saddle_point_problem} alternatively.
\begin{remark}
  Recall that the solution set of \cref{mP} is nonempty, thus \cref{existence_saddle_point} holds if the strong duality holds for \cref{mP}, which can be guaranteed by certain constraint qualifications. A popular one is Slater's condition \cite{boyd2004convex}, which, when applied to \cref{1613}, states that the strong duality holds for \cref{mP} if \cref{mP} is convex and there exists a point $\mx \in \mathbf{relint}(\mathcal{D}(g))$ such that $A\mx \in \mathbf{relint}(\mathcal{D}(h))$, where $\mathcal{D}(\cdot)$ denotes the domain of a function and $\mathbf{relint}(\cdot)$ represents the relative interior \cite{boyd2004convex} of a set. Notice that $\mathbf{relint}(\mR^p) = \mR^p$, hence \cref{existence_saddle_point} naturally holds if $\mathbf{relint}(\mathcal{D}(g)) \neq \emptyset$ and $\mathcal{D}(h) = \mR^p$.
\end{remark}
\begin{assumption} \label{G}
  The network topology among agents is an undirected and connected graph.
\end{assumption}

\begin{assumption} \label{convex}
  The cost functions satisfy
  \begin{enumerate}[(i)]
    \item $f_i$ is $\mu_i$-strongly convex and $l_i$-smooth for $l_i \geq \mu_i > 0$, $\forall i \in \mathcal{N}$;
    \item $g_i$ and $h$ are proper convex and lower semi-continuous, $\forall i \in \mathcal{N}$.
  \end{enumerate}
\end{assumption}
According to \cref{convex}, it is obvious that $f$ is $\mu$-strongly convex and $l$-smooth, where $\mu = \min_{i \in \mathcal{N}}{\mu_i}$ and $l = \max_{i \in \mathcal{N}}{l_i}$.
As an efficient method for solving saddle-point problems, the primal-dual (proximal) gradient algorithm (PGA) has various versions \cite{arrow1958studies,zhu2008efficient,chambolle2011first}. In this work, we consider the one proposed in \cite{chambolle2011first}, which is given as
\eqe{ \label{pga}
  \mx\+     & = \prox_{\alpha g}\lt(\mx^k - \alpha(\nabla f(\mx^k) + A\T\lambda^k)\rt), \\
  \lambda\+ & = \prox_{\beta h^*}\lt(\lambda^k + \beta A\lt(\mx\+ + \theta(\mx\+ - \mx^k)\rt)\rt),
}
where $\alpha, \beta > 0$ and $\theta \geq 0$.
However, we cannot directly employ PGA to solve \cref{mP} in a decentralized manner, since the update of $\lambda$ requires the globally coupled term $A\mx\+=\sum_{i=1}^{n}A_ix_i\+$. To overcome this challenge, we propose an efficient decentralized variant of PGA, called NPGA, based on the dual perspective.

We begin by observing that the dual function of \cref{mP} can be decomposed as
\eqe{
  \phi(\lambda) &= \inf_{\mx \in \mR^d} f(\mx) + g(\mx) + \lambda\T A\mx - h^*(\lambda) \\
  &= \sum_{i=1}^n\lt(\inf_{x_i \in \mR^{d_i}} f_i(x_i) + g_i(x_i) + \lambda\T A_ix_i\rt) - h^*(\lambda) \\
  &= \sum_{i=1}^n \phi_i(\lambda) - h^*(\lambda),
  \nonumber
}
where $\phi_i(\lambda) = \inf_{x_i \in \mR^{d_i}} f_i(x_i) + g_i(x_i) + \lambda\T A_ix_i$.
Consequently, we can reformulate the dual problem of \cref{mP} as
\eqe{ \label{dual1}
  \max_{\lambda \in \mR^{p}} \ \sum_{i=1}^n \phi_i(\lambda) - h^*(\lambda),
}
which corresponds to the classical decentralized unconstrained optimization problem \cref{uc} (with an extra term $h^*(\lambda)$ that is always convex but possibly nonsmooth). It is well known that \cref{uc} has been widely studied in recent years and numerous effective algorithms has been proposed to solve it, including DIGing \cite{qu2017harnessing,nedic2017achieving}, exact first-order algorithm (EXTRA) \cite{shi2015extra}, decentralized linearized alternating direction method of multipliers (DLM) \cite{ling2015dlm}, proximal primal-dual diffusion (P2D2) \cite{alghunaim2019linearly}, augmented distributed gradient
methods (Aug-DGM) \cite{xu2015augmented,nedic2017geometrically}, ATC tracking \cite{di2016next,scutari2019distributed}, exact diffusion \cite{yuan2018exact}, network independent step-size (NIDS) \cite{li2019decentralized}.

Note that \cref{existence_saddle_point} implies that the strong duality holds for \cref{mP}, thus we can employ numerous distributed unconstrained optimization algorithms to solve \cref{dual1}. Once \cref{dual1} is solved, solving \cref{mP} becomes straightforward. The question lies in selecting the most appropriate algorithm to solve \cref{dual1} among the existing distributed unconstrained optimization algorithms. Although any algorithm is theoretically feasible, the performance of different algorithms may vary significantly depending on the specific scenario. Therefore, we are interested in a general primal-dual algorithm framework called the proximal unified decentralized algorithm (PUDA) \cite{alghunaim2020decentralized}.
PUDA offers flexibility by allowing the selection of different network matrices $B \in \mathbb{R}^{n \times n}$, $C \in \mathbb{R}^{n \times n}$, and $D \in \mathbb{R}^{n \times n}$, which are associated with the network topology among the agents. By appropriately choosing these matrices, PUDA can recover many existing distributed unconstrained optimization algorithms, as shown in \cref{tab1}. The detailed derivation of the equivalence between PUDA and existing algorithms can be found in \cite{alghunaim2020decentralized}.
\begin{remark} \label{rem1}
  For an undirected and connected graph, its corresponding mixing matrix $W=[w_{ij}]\in \mR^{n\times n}$ satisfies \cite{shi2015extra}
  \begin{enumerate}[(i)]
    \item $w_{ij}>0$ if $i = j$ or $(i,j) \in \mathcal{E}$, otherwise $w_{ij} = 0$;
    \item $W=W\T$;
    \item $\Null(\mI-W) = \Span(\1_n)$.
  \end{enumerate}
  Given the above three properties, we can easily verify that $-\mI < W \leq \mI$ \cite{shi2015extra}. There are various methods to construct a mixing matrix $W$ satisfying these properties, and one commonly used approach is the Laplacian method \cite{xiao2004fast}:
  $$W = \mI-\frac{L}{\tau},$$
  where $\tau=\max_{i\in \mathcal{N}}e_i+c$, $e_i$ represents the degree of agent $i$, and $c>0$ is a constant.
\end{remark}

\begin{table*}[t]
  \caption{Different versions of NPGA. The parameters used in this table correspond to those in \cref{tab1}.}
  \label{tab2}
  \centering
  \begin{tabular}{ccccc}
    \toprule
    Versions             & $B^2$                & $C$                  & $D$                  & Communication rounds \\
    \midrule
    NPGA-DIGing          & $(\mI-W)^2$          & $\mI-W^2$            & $\mI$                & 2                    \\
    NPGA-EXTRA           & $\frac{1}{2}(\mI-W)$ & $\frac{1}{2}(\mI-W)$ & $\mI$                & 1                    \\
    NPGA-DLM             & $c\beta L$           & $c\beta L$           & $\mI$                & 1                    \\
    NPGA-P2D2            & $\frac{c}{2}(\mI-W)$ & $\frac{1}{2}(\mI-W)$ & $\mI$                & 1                    \\
    NPGA-Aug-DGM         & $(\mI-W)^2$          & $0$                  & $W^2$                & 2                    \\
    NPGA-ATC tracking    & $(\mI-W)^2$          & $\mI-W$              & $W$                  & 2                    \\
    NPGA-Exact diffusion & $\frac{1}{2}(\mI-W)$ & $0$                  & $\frac{1}{2}(\mI+W)$ & 1                    \\
    NPGA-NIDS            & $c(\mI-W)$           & $0$                  & $\mI-c(\mI-W)$       & 1                    \\
    NPGA-I               & $\mI-W$              & $0$                  & $W^2$                & 2                    \\
    NPGA-II              & $\mI-W$              & $\mI-W$              & $W$                  & 2                    \\
    \bottomrule
  \end{tabular}
\end{table*}

In the following, we will show how NPGA is designed based on a variant of PUDA.
Define $\ml = [\lambda_1\T, \cdots, \lambda_n\T]\T \in \mR^{np}$, we can easily verify that \cref{dual1} is equivalent to
\eqe{ \label{dual2}
  \max_{\ml \in \mR^{np}} \ &\mphi(\ml) - \mh(\ml), \\
  \text{s.t.} \ & \lambda_i = \lambda_j, \ i = 1, \cdots, n,
}
where $\mphi(\ml) = \sum_{i=1}^n \phi_i(\lambda_i)$ and $\mh(\ml) = \frac{1}{n} \sum_{i=1}^n h^*(\lambda_i)$. Let $B \in \mR^{n\times n}$, $C \in \mR^{n\times n}$, and $D \in \mR^{n\times n}$ be symmetric matrices associated with the network topology among agents. For \cref{dual2}, our variant of PUDA is implemented as
\eqe{ \label{puda}
  \mv\+ =& \ml^k - \mC\ml^k - \mB\my^k \\
  &+\beta\lt(\nabla\mphi(\ml^k) + \theta\lt(\nabla\mphi(\ml^k)-\nabla\mphi(\ml^{k-1})\rt)\rt), \\
  \my\+ =& \my^k + \gamma \mB\mv\+, \\
  \ml\+ =& \prox_{\beta\mh}\lt(\mD\mv\+\rt),
}
where $\mB = B \otimes \mI_p$, $\mC = C \otimes \mI_p$, $\mD = D \otimes \mI_p$, $\beta, \gamma > 0$, and $\theta \geq 0$.
\begin{remark}
  The differences between \cref{puda} and the original PUDA are as follows:
  \begin{enumerate}[(i)]
    \item Inspired by the PGA proposed in \cite{chambolle2011first}, we replace $\nabla\mphi(\ml^k)$ with $\nabla\mphi(\ml^k) + \theta\lt(\nabla\mphi(\ml^k)-\nabla\mphi(\ml^{k-1})\rt)$ in the primal update, which is quite useful to improve the actual performance of NPGA.
    \item We introduce a step-size $\gamma$ in the dual update, while in the original PUDA, it is always set to $1$. In \cref{convergence}, we will see that the condition $\gamma < 1$ is crucial to guarantee the linear convergence of NPGA under weaker conditions.
  \end{enumerate}
\end{remark}

Define
$$x_i^*(\lambda_i) = \arg \min_{x_i \in \mR^{d_i}}\Big\{f_i(x_i) + g_i(x_i) + \lambda_i\T A_ix_i\Big\},$$
when $f_i$ is strictly convex, $x_i^*(\lambda_i)$ is always unique, then $\phi_i$ is differentiable and \cite{bertsekas1997nonlinear}
\eqe{
  \nabla \phi_i(\lambda_i) = A_ix_i^*(\lambda_i).
  \nonumber
}
As a result, $\mphi$ is differentiable, and we have
\eqe{
  \nabla \mphi(\ml) =& [\nabla \phi_1(\lambda_1)\T, \cdots, \nabla \phi_n(\lambda_n)\T]\T \\
  =& \mA\mx^*(\ml),
  \nonumber
}
where $\mA = \text{diag}(A_1, \cdots, A_n)$. Let
\eqe{
  \mx^*(\ml) =& [x_1^*(\lambda_1)\T, \cdots, x_n^*(\lambda_n)\T]\T \\
  =& \arg \min_{\mx \in \mR^{d}}\lt\{f(\mx) + g(\mx) + \ml\T\mA\mx\rt\},
  \nonumber
}
then \cref{puda} can be unfolded as
\eqe{ \label{da2}
  \mx^*(\ml^k) =& \arg \min_{\mx \in \mR^{d}}\lt\{f(\mx) + g(\mx) + \ml^{k\T}\mA\mx\rt\}, \\
  \mv\+ =& \ml^k - \mC\ml^k - \mB\my^k \\
  &+ \beta\mA\lt(\mx^*(\ml^k)+\theta\lt(\mx^*(\ml^k)-\mx^*(\ml^{k-1})\rt)\rt), \\
  \my\+ =& \my^k + \gamma \mB\mv\+, \\
  \ml\+ =& \prox_{\beta\mh}\lt(\mD\mv\+\rt).
}
At each iteration of \cref{da2}, we have to solve a subproblem exactly to obtain the dual gradient $\mA\mx^*(\ml^k)$. However, solving an optimization problem exactly is often costly and sometimes even impossible \cite{devolder2014first}. To address this issue, an inexact version of \cref{da2}, i.e., NPGA, is further designed:
\begin{subequations} \label{npga}
  \begin{align}
    \mx\+       & = \prox_{\alpha g}\lt(\mx^k - \alpha(\nabla f(\mx^k) + \mA\T\ml^k)\rt), \label{npga_x} \\
    \mv\+       & = \ml^k - \mC\ml^k - \mB\my^k + \beta\mA\hat{\mx}\+, \label{npga_v}                    \\
    \my\+       & = \my^k + \gamma \mB\mv\+, \label{npga_z}                                              \\
    \ml\+       & = \prox_{\beta\mh}\lt(\mD\mv\+\rt), \label{npga_l}                                     \\
    \hat{\mx}\+ & = \mx\+ + \theta(\mx\+ - \mx^k),
  \end{align}
\end{subequations}
where we replace the exact dual gradient with an inexact one updated by single-step proximal gradient descent.
Therefore, NPGA can be seen as an inexact version of \cref{da2}. Alternatively, NPGA can also be seen as a decentralized version of PGA, where \cref{npga_x} and \cref{npga_v}-\cref{npga_l} corresponds to the primal and dual updates of PGA, respectively. Also note that \cref{npga_v}-\cref{npga_l} can be regraded as a special case of PGA, thus NPGA includes two PGAs and one is nested in the other one, which is why it is named NPGA.

\begin{remark}
  As mentioned before, by choosing different $B$, $C$, and $D$, PUDA can recover many existing distributed unconstrained optimization algorithms. As an inexact dual version of PUDA, NPGA does inherit this property of PUDA, and we can also obtain many different versions of NPGA by choosing different $B$, $C$, and $D$, as shown in \cref{tab2}, where the combinations of matrices for  NPGA-I and NPGA-II are proposed in the experiment section of \cite{alghunaim2020decentralized}. It is also worthy noting that both DCDA and DCPA are special cases of NPGA: DCDA corresponds to NPGA-Exact diffusion with $\theta=0$ and $\gamma=1$, while DCPA corresponds to NPGA-P2D2 with $c=1$ and $\theta=1$.
  In numerical experiments, it has been shown that the convergence rates of different versions of NPGA usually have significant differences. This indicates that NPGA could provide an opportunity to design more efficient distributed algorithms for \cref{pro2}.
\end{remark}

\begin{remark}
  To illustrate the decentralized implementation of NPGA in practice, we can take NPGA-EXTRA as an example.
  Notice that NPGA can be rewritten as
  \eqe{ \label{npga2}
    \mx\+ =& \prox_{\alpha g}\lt(\mx^k - \alpha(\nabla f(\mx^k) + \mA\T\ml^k)\rt), \\
    \mv\+ =& (\mI - \mC)(\ml^k-\ml^{k-1}) + (\mI - \gamma \mB^2)\mv^k \\
    &+ \beta\mA(\hat{\mx}\+-\hat{\mx}^k), \\
    \ml\+       =& \prox_{\beta\mh}\lt(\mD\mv\+\rt), \\
    \hat{\mx}\+ =& \mx\+ + \theta(\mx\+ - \mx^k),
  }
  by substituting the specific choice of $B^2$, $C$, and $D$ for NPGA-EXTRA (as listed in \cref{tab2}) into \cref{npga2}, we have
  \eqe{
    \mv\+ =& \frac{\mI+\mW}{2}\lt(\ml^k-\ml^{k-1}\rt) + \frac{(2-\gamma)\mI+\gamma\mW}{2}\mv^k \\
    &+\beta\mA(\hat{\mx}\+-\hat{\mx}^k), \\
    \ml\+ =& \prox_{\beta\mh}\lt(\mv\+\rt),
    \nonumber
  }
  where $\mW=W\otimes \mI_p$. Consequently, NPGA-EXTRA can be implemented as
  \eqe{
  x_i\+ =& \prox_{\alpha g_i}\lt(x_i^k - \alpha(\nabla f_i(x_i^k)+A_i\T\lambda_i^k)\rt), \\
  v_i\+ =& \frac{1}{2}(\lambda_i^k - \lambda_i^{k-1}) + \frac{1}{2}\sum_{j=1}^nw_{ij}\lt(\lambda_j^k - \lambda_j^{k-1} + \gamma v_j^k\rt) \\
  &+ \frac{2-\gamma}{2}v_i^k + \beta A_i(\hat{x}_i\+-\hat{x}_i^k), \\
  \lambda_i\+ =& \prox_{\frac{\beta}{n}h^*}\lt(v_i\+\rt), \\
  \hat{x}_i\+ =& x_i\+ + \theta(x_i\+ - x_i^k),
  \nonumber
  }
  for agent $i$, recall that $w_{ij}=0$ if $i \neq j$ and agent $j$ is not the neighbor of agent $i$. Also note that agent $i$ only needs to send $\lambda_i^k - \lambda_i^{k-1} + \gamma v_i^k$ to its neighbors, hence NPGA-EXTRA requires only one round of communication per iteration. The decentralized implementations and the corresponding communication rounds of other versions of NPGA can be derived in a similar way.
\end{remark}

\section{Convergence Analysis} \label{convergence}
In this section, we analyze the linear convergence of NPGA under two cases: (i) $g_i=0$ and $A$ has full row rank; (ii) $h$ is $l_h$-smooth.
Before proceeding with the analysis, we introduce a necessary assumption that applies to both cases.
\begin{assumption} \label{condition_main_inequality}
  The network matrices $B$, $C$, and $D$ satisfy
  \begin{enumerate}[(i)]
    \item $C = 0$ or $\Null(C) = \Span(\1_n)$;
    \item $\Null(B) = \Span(\1_n)$;
    \item $0 \leq C < \mI$, $B^2 \leq \mI$;
    \item $D$ is a symmetric doubly stochastic matrix.
  \end{enumerate}
\end{assumption}

\begin{remark} \label{remark_condition_main_inequality}
  With the three properties of $W$ given in \cref{rem1}, $\Null(L) = \Span(\1_n)$, and $\Null(B) = \Span(\1_n)$ is equivalent to $\Null(B^2) = \Span(\1_n)$, we can easily verify that all algorithms listed in \cref{tab2} satisfy (i), (ii), and (iv). Since $L \geq 0$, we can choose small enough $c$ and $\beta$ such that NPGA-DLM satisfies (iii).
  Note that $-\mI < W \leq \mI$, then we have $0 \leq \mI - W < 2\mI$, hence NPGA-EXTRA and NPGA-Exact diffusion both satisfy (iii), NPGA-P2D2 and NPGA-NIDS satisfy (iii) for $c \leq 1$ and $c \leq \frac{1}{2}$, respectively. Let $W' \triangleq \frac{1}{2}(\mI+W)$, obviously $0 < W' \leq \mI$. By replacing $W$ with $W'$, all the remaining algorithms can satisfy (iii).
\end{remark}

The following lemma establishes the relation between the fixed point of NPGA and the solution of the saddle-point problem \cref{saddle_point_problem}.
\begin{lemma} \label{fixed_point}
  Assume \cref{condition_main_inequality} holds, $(\mx^*, \lambda^*)$ is a solution of the saddle-point problem \cref{saddle_point_problem} if and only if there exist $\mv^* \in \mR^{np}$ and $\my^* \in \mR^{np}$ such that $(\mx^*, \mv^*, \my^*, \ml^*)$ is a fixed point of NPGA, where $\ml^* = \1_n \otimes \lambda^*$.
\end{lemma}
\begin{proof}
  See Appendix \ref{appendix_fixed_point}.
\end{proof}

The optimality condition of the saddle-point problem \cref{saddle_point_problem} is given by
\begin{subequations} \label{opt_condition}
  \begin{align}
    0      & \in \nabla f(\mx^*) + A\T\lambda^* + \partial g(\mx^*), \label{opt_condition_a} \\
    A\mx^* & \in \partial h^*(\lambda^*), \label{opt_condition_b}
  \end{align}
\end{subequations}
then we can verify that the uniqueness of the solution of \cref{saddle_point_problem} can be guaranteed by combining the strong convexity of $f$ with the condition that $A$ has full row rank or $h^*$ is strongly convex, which implies that the solution of \cref{saddle_point_problem} is unique for both cases.
Let $(\mx^*, \mv^*, \my^*, \ml^*)$ be a fixed point of \cref{npga}, we can easily verify that $\mx^*$, $\mv^*$, and $\ml^*$ are unique, but $\my^*$ is not necessarily unique due to the singularity of $\mB$. Nevertheless, $\my^*$ is indeed unique in $\col(\mB)$, denoted as $\my^*_c$. If $\my^0=0$, then we have $\my^k \in \col(\mB)$ for all $k \geq 0$, which implies that, given the premise that NPGA converges, $\my^k$ can only converge to $\my^*_c$. Consequently, we only need to prove that NPGA linearly converge to $(\mx^*, \mv^*, \my^*_c, \ml^*)$ given the initial condition $\my^0=0$.

Define the error variables as
\eqe{
  \tx^k =& \mx^k-\mx^*, \\
  \tv^k =& \mv^k-\mv^*, \\
  \ty^k =& \my^k-\my^*_c, \\
  \tw^k =& \ml^k-\ml^*,
  \nonumber
}
then the error system of NPGA can be written as
\eqe{ \label{npga_error}
  \tx\+ =& \tx^k - \alpha(\nabla f(\mx^k)-\nabla f(\mx^*)) - \alpha\mA\T\tw^k \\
  &- \alpha\lt(\mS_g(\mx\+)-\mS_g(\mx^*)\rt), \\
  \tv\+ =& \tw^k - \mC\tw^k - \mB\ty^k + \beta \mA\lt(\tx\+ + \theta(\tx\+-\tx^k)\rt), \\
  \ty\+ =& \ty^k + \gamma \mB\tv\+, \\
  \tw\+ =& \prox_{\beta\mh}\lt(\mD\mv\+\rt) - \prox_{\beta\mh}\lt(\mD\mv^*\rt).
}
The following convergence analysis will be based on the above error system.

\subsection{Case I: $g_i=0$ and $A$ Has Full Row Rank}
\begin{assumption} \label{A}
  $g_i=0$ and $A$ has full row rank.
\end{assumption}
\begin{remark}
  As mentioned before, \cref{A} is the weakest known condition for decentralized algorithms to linearly solve \cref{mP} with a non-smooth $h$. Before NPGA, there are two algorithms that can achieve linear convergence under the condition that $A$ has full row rank: (i) DCPA, which is a special case of NPGA; (ii) IDEA, which considers a special case of \cref{mP}, i.e., $h$ is an indicator function corresponding to the equality constraint.
\end{remark}

Since $g_i = 0$, the error system \cref{npga_error} can be simplified to
\eqe{ \label{npga_error_g0}
  \tx\+ =& \tx^k - \alpha\lt(\nabla f(\mx^k)-\nabla f(\mx^*)\rt) - \alpha\mA\T\tw^k, \\
  \tv\+ =& \tw^k - \mC\tw^k - \mB\ty^k + \beta\mA\lt(\tx\+ + \theta(\tx\+-\tx^k)\rt), \\
  \ty\+ =& \ty^k + \gamma \mB\tv\+, \\
  \tw\+ =& \prox_{\beta\mh}\lt(\mD\mv\+\rt) - \prox_{\beta\mh}\lt(\mD\mv^*\rt).
}

In order to establish the linear convergence of NPGA under the condition that $A$ has full row rank, we need to introduce the following key lemma.
\begin{lemma} \label{mA}
  Assume \cref{A} holds, $M \in \mR^{n \times n}$ is a symmetric double stochastic matrix and $H \in \mR^{n \times n}$ is a positive semi-definite matrix which satisfies $\Null(H) = \Span(\1_n)$, then $\mathbf{M}\mA\mA\T\mathbf{M}+c\mathbf{H} > 0$, where $\mathbf{M} = M \otimes \mI_p$, $\mathbf{H} = H \otimes \mI_p$, and $c>0$ is a constant.
\end{lemma}
\begin{proof}
  See Appendix \ref{appendix_mA}.
\end{proof}

We now present an inequality that plays a crucial role in the subsequent analysis.
\begin{lemma} \label{main_inequality}
  Given \cref{condition_main_inequality,A,existence_saddle_point,G,convex}, let $\my^0=0$, $\theta \geq 0$, and the positive step-sizes satisfy
  \eqe{ \label{step_npga_g0}
    \alpha < \frac{1}{l(1+2\theta)}, \ \beta \leq \frac{\mu}{\os^2(\mA)\lt(\frac{1}{1-\os(C)}+\theta\rt)}, \ \gamma < 1,
  }
  then we have
  \eqe{
  &c_1\norm{\tx\+}^2 + c_2\norm{\tv\+}^2_{\mI-\gamma \mB^2} + c_3\norm{\ty\+}^2 \\
  \leq& \delta_1c_1\norm{\tx^k}^2 + c_2\norm{\tw^k}^2_{\mI-\mC-\alpha\beta \mA\mA\T} \\
  &+ 2\theta \alpha\beta\os(C)\os^2(\mA)c_2\norm{\tw^k}^2_{\mC} + \delta_3c_3\norm{\ty^k}^2,
  }
  where $c_1 = 1-\alpha\beta\os^2(\mA)\lt(\frac{1}{1-\os(C)}+\theta\rt) > 0$, $c_2 = \frac{\alpha}{\beta}$, $c_3 = \frac{\alpha}{\beta\gamma}$, $\delta_1 = 1-\alpha\mu(1-\alpha l(1+2\theta)) \in (0, 1)$, $\delta_3 = 1-\gamma\us^2(B) \in (0, 1)$.
\end{lemma}
\begin{proof}
  See Appendix \ref{appendix_main_inequality}.
\end{proof}

To obtain the linear convergence of NPGA, we need to impose another condition on $C$, i.e., $\Null(C) = \Span(\1_n)$, which excludes the versions of NPGA with $C = 0$.
\begin{assumption} \label{null_C}
  The network matrix $C$ satisfies
  $$\Null(C) = \Span(\1_n).$$
\end{assumption}

Based on \cref{main_inequality} and \cref{null_C}, we can obtain the following result.
\begin{lemma} \label{main_inequality_C}
  Under the same condition with \cref{main_inequality}, and assume \cref{null_C} holds, then we have
  \eqe{
  &c_1\norm{\tx\+}^2 + c_2\norm{\tv\+}^2_{\mI-\gamma \mB^2} + c_3\norm{\ty\+}^2 \\
  \leq& \delta_1c_1\norm{\tx^k}^2 + \delta_2c_2\norm{\tw^k}^2 + \delta_3c_3\norm{\ty^k}^2,
  }
  where $E=\mA\mA\T+\frac{1}{2\alpha\beta}\mC$, $\delta_2 = 1-\alpha\beta\ue(E) \in (0, 1)$, $c_1$, $c_2$, $c_3$, $\delta_1$, and $\delta_3$ are defined in \cref{main_inequality}.
\end{lemma}
\begin{proof}
  See Appendix \ref{appendix_main_inequality_C}.
\end{proof}

Now we present the linear convergence result of NPGA under Case I.
\begin{theorem} \label{convergence_npga_g0}
  Given \cref{condition_main_inequality,A,existence_saddle_point,G,convex,null_C}, let $\my^0=0$, $\theta \geq 0$, and the positive step-sizes satisfy
  \eqe{ \label{step_theorem_1}
    \alpha < \frac{1}{l(1+2\theta)}, \ \beta \leq \frac{\mu}{\os^2(\mA)\lt(\frac{1}{1-\os(C)}+\theta\rt)}, \\
    \gamma < \min\lt\{1, \frac{\alpha\beta\ue(E)}{\os^2(B)}\rt\},
  }
  then there exists $\delta \in (0, 1)$ such that
  \eqe{ \label{R-linear}
    \norm{\mx^k-\mx^*}^2 = \mathcal{O}\lt(\delta^k\rt),
  }
  where
  \eqe{
    \delta = \max\Bigg\{1-\alpha\mu(1-\alpha l(1+2\theta)), \frac{1-\alpha\beta\ue(E)}{1-\gamma\os^2(B)}, \\1-\gamma\us^2(B)\Bigg\}.
    \nonumber
  }
\end{theorem}

\begin{proof}
  Recall that $\mD$ is a symmetric doubly stochastic matrix, which implies that $\|\mD\|<1$, then applying the non-expansive property of the proximal operator gives that
  \eqe{
    \norm{\tw\+}^2 =& \norm{\prox_{\beta\mh}\lt(\mD\mv\+\rt) - \prox_{\beta\mh}\lt(\mD\mv^*\rt)}^2 \\
    \leq& \norm{\mD\tv\+}^2 \\
    \leq& \norm{\tv\+}^2.
    \nonumber
  }
  Note that the condition of \cref{convergence_npga_g0} is sufficient for \cref{main_inequality_C}, we have
  \eqe{
  &c_1\norm{\tx\+}^2 + \lt(1-\gamma\os^2(B)\rt)c_2\norm{\tw\+}^2 + c_3\norm{\ty\+}^2 \\
  \leq& c_1\norm{\tx\+}^2 + \lt(1-\gamma\os^2(B)\rt)c_2\norm{\tv\+}^2 + c_3\norm{\ty\+}^2 \\
  \leq& c_1\norm{\tx\+}^2 + c_2\norm{\tv\+}^2_{\mI-\gamma \mB^2} + c_3\norm{\ty\+}^2 \\
  \leq& \delta_1c_1\norm{\tx^k}^2 + \delta_2c_2\norm{\tw^k}^2 + \delta_3c_3\norm{\ty^k}^2,
  }
  let $\omega = 1-\gamma\os^2(B)$, obviously $\omega \in (0, 1)$, we then have
  \eqe{
    &\frac{c_1}{\omega}\norm{\tx\+}^2 + c_2\norm{\tw\+}^2 + \frac{c_3}{\omega}\norm{\ty\+}^2 \\
    \leq& \delta_1\frac{c_1}{\omega}\norm{\tx^k}^2 + \frac{\delta_2}{\omega}c_2\norm{\tw^k}^2 + \delta_3\frac{c_3}{\omega}\norm{\ty^k}^2 \\
    \leq& \delta\lt(\frac{c_1}{\omega}\norm{\tx^k}^2 + c_2\norm{\tw^k}^2 + \frac{c_3}{\omega}\norm{\ty^k}^2\rt),
  }
  where $\delta = \max\lt(\delta_1, \frac{\delta_2}{\omega}, \delta_3\rt)$. It follows that
  \eqe{
    \frac{c_1}{\omega}\norm{\tx^k}^2 \leq \delta^k\lt(\frac{c_1}{\omega}\norm{\tx^0}^2 + c_2\norm{\tw^0}^2 + \frac{c_3}{\omega}\norm{\ty^0}^2\rt),
  }
  the step-sizes condition \cref{step_npga_g0} implies that $\frac{\delta_2}{\omega} \in (0, 1)$, thus $\delta \in (0, 1)$, which completes the proof.
\end{proof}

\begin{remark}
  Given a sequence $\{x^k\}$ that converges to $x^*$, the convergence is called (i) R-linear if there exists $\lambda \in (0, 1)$ such that $\norm{x^k - x^*} \leq C\lambda^k$ for $k \geq 0$, where $C$ is a positive constant; (ii) Q-linear if there exists $\lambda \in (0, 1)$ such that $\frac{\norm{x^{k+1} - x^*}}{\norm{x^k - x^*}} \leq \lambda$ for $k \geq 0$ \cite{nedic2017achieving}.
  Therefore, \cref{R-linear} implies that $\mx^k$ converges to $\mx^*$ with an R-linear rate of $\sqrt{\delta}$.
\end{remark}

If the following assumption is further made, we can obtain a tighter convergence rate than \cref{convergence_npga_g0}.
\begin{assumption} \label{DB}
  The network matrices $B$ and $D$ satisfy
  $$D^2 \leq \mI - B^2.$$
\end{assumption}

\begin{remark}
  We can easily see that all CTA versions of NPGA (where $D = \mI$) cannot satisfy \cref{DB}, only ATC ones have the possibility. Since \cref{convergence_npga_g0_DB} requires both \cref{null_C,DB}, we only analyze ATC versions of NPGA with $C \neq 0$, i.e., NPGA-ATC tracking and NPGA-II. Recall that $W \leq \mI$, which implies that $W^2 \leq W$, then we can easily verify that NPGA-ATC tracking and NPGA-II both satisfy \cref{DB} according to \cref{tab2}.
\end{remark}

\begin{theorem} \label{convergence_npga_g0_DB}
  Given \cref{condition_main_inequality,A,existence_saddle_point,G,convex,null_C,DB}, let $\my^0=0$, $\theta \geq 0$, and the positive step-sizes satisfy
  \eqe{ \label{step_npga_g0_DB}
    \alpha < \frac{1}{l(1+2\theta)}, \ \beta \leq \frac{\mu}{\os^2(\mA)\lt(\frac{1}{1-\os(C)}+\theta\rt)}, \ \gamma < 1,
  }
  then there exists $\delta \in (0, 1)$ such that
  $$\norm{\mx^k-\mx^*}^2 = \mathcal{O}\lt(\delta^k\rt),$$
  where
  \eqe{
    \delta = \max\lt\{1-\alpha\mu(1-\alpha l(1+2\theta)), 1-\alpha\beta\ue(E), 1-\gamma\us^2(B)\rt\}.
    \nonumber
  }
\end{theorem}

\begin{proof}
  According to \cref{DB}, we have
  $$\mD^2 \leq \mI-\mB^2 \leq \mI-\gamma\mB^2,$$
  it follows that
  \eqe{
  \norm{\tw\+}^2 =& \norm{\prox_{\beta\mh}\lt(\mD\mv\+\rt) - \prox_{\beta\mh}\lt(\mD\mv^*\rt)}^2 \\
  \leq& \norm{\mD\tv\+}^2 \\
  \leq& \norm{\tv\+}^2_{\mI-\gamma\mB^2}.
  \nonumber
  }
  Note that the condition of \cref{convergence_npga_g0_DB} is sufficient for \cref{main_inequality_C}, we have
  \eqe{
  &c_1\norm{\tx\+}^2 + c_2\norm{\tw\+}^2 + c_3\norm{\ty\+}^2 \\
  \leq& c_1\norm{\tx\+}^2 + c_2\norm{\tv\+}^2_{\mI-\gamma \mB^2} + c_3\norm{\ty\+}^2 \\
  \leq& \delta_1c_1\norm{\tx^k}^2 + \delta_2c_2\norm{\tw^k}^2 + \delta_3c_3\norm{\ty^k}^2 \\
  \leq& \delta\lt(c_1\norm{\tx^k}^2 + c_2\norm{\tw^k}^2 + c_3\norm{\ty^k}^2\rt),
  \nonumber
  }
  where $\delta = \max\{\delta_1,\delta_2,\delta_3\} \in (0, 1)$. It follows that
  \eqe{
    c_1\norm{\tx^k}^2 \leq \delta^k\lt(c_1\norm{\tx^0}^2 + c_2\norm{\tw^0}^2 + c_3\norm{\ty^0}^2\rt),
  }
  which completes the proof.
\end{proof}

\begin{remark}
  Compared to \cref{convergence_npga_g0}, the improved convergence result in this case relaxes the upper bound on $\gamma$ to $1$ and provides a tighter convergence rate of $\max\lt\{1-\alpha\mu(1-\alpha l(1+2\theta)), 1-\alpha\beta\ue(E), 1-\gamma\us^2(B)\rt\}$. It is worth noting that the convergence of DCPA under Case I can be encompassed by \cref{convergence_npga_g0} with negligible differences. Therefore, we can conclude that the ATC versions of NPGA, which satisfy \cref{DB}, achieve tighter convergence rates than DCPA.
\end{remark}

\cref{convergence_npga_g0,convergence_npga_g0_DB} both assume \cref{null_C} holds, which cannot be satisfied by the versions of NPGA with $C=0$. Fortunately, we find that the linear convergence of NPGA can be guaranteed without \cref{null_C}, when $h$ is an indicator function that corresponds to the equality constraint.
\begin{assumption} \label{indicator}
  $h$ is an indicator function given as
  $$h(y)=\lt\{
    \begin{aligned}
      0,      \   & y = b,    \\
      +\infty, \  & y \neq b,
    \end{aligned}\right.$$
  where $b \in \mR^p$ is a const vector.
\end{assumption}

Given \cref{indicator}, it holds that
\eqe{
  h^*(\lambda) =& b\T\lambda, \\
  \prox_{\beta\mh}(\ml) =& \ml - \beta \mb,
  \nonumber
}
where $\mb = \frac{1}{n}\1_n \otimes b$, then the error system \cref{npga_error_g0} can be simplified to
\eqe{ \label{npga_error_g0_indicator}
  \tx\+ =& \tx^k - \alpha\lt(\nabla f(\mx^k)-\nabla f(\mx^*)\rt) - \alpha\mA\T\tw^k, \\
  \tv\+ =& \tw^k - \mC\tw^k - \mB\ty^k + \beta \mA\lt(\tx\+ + \theta(\tx\+-\tx^k)\rt), \\
  \ty\+ =& \ty^k + \gamma \mB\tv\+, \\
  \tw\+ =& \mD\tv\+.
}

Different from \cref{convergence_npga_g0,convergence_npga_g0_DB}, we replace \cref{null_C} with the following much weaker assumption, which can be satisfied by all versions of NPGA listed in \cref{tab2}.
\begin{assumption} \label{DCB}
  The network matrices $B$, $C$, and $D$ satisfy
  $$D(\mI-C)D \leq \mI - B^2.$$
\end{assumption}

\begin{remark}
  We analyze \cref{DCB} for the ATC and CTA versions of NPGA separately:
  \begin{enumerate}[(i)]
    \item For CTA versions (where $D=\mI$), \cref{DCB} simplifies to $B^2 \leq C$. Obviously NPGA-EXTRA, NPGA-DLM satisfy this condition, and NPGA-P2D2 can satisfy it by choosing $c \leq 1$. Recall that $W \leq \mI$, which implies that $W^2 \leq W$, hence DIGing also satisfies \cref{DCB}.
    \item By utilizing the properties $W \leq \mI$ and $W^2 \leq W$, we can easily verify that all the ATC versions of NPGA listed in \cref{tab2} satisfy \cref{DCB}, where NPGA-NIDS needs to choose $c \leq \frac{1}{2}$ to satisfy the condition.
  \end{enumerate}
\end{remark}

\begin{theorem} \label{convergence_npga_g0_indicator}
  Given \cref{condition_main_inequality,A,existence_saddle_point,G,convex,indicator,DCB}, let $\my^0=0$, $\theta = 0$, and the positive step-sizes satisfy
  \eqe{ \label{step_npga_g0_indicator}
    \alpha < \frac{1}{l}, \ \beta \leq \frac{\mu(1-\os(C))}{\os^2(\mA)}, \ \gamma < 1,
  }
  then there exists $\delta \in (0, 1)$ such that
  $$\norm{\mx^k-\mx^*}^2 = \mathcal{O}\lt(\delta^k\rt),$$
  where
  \eqe{
    \delta = \max\lt\{1-\alpha\mu(1-\alpha l), 1-\alpha\beta\ue(F), 1-\gamma\us^2(B)\rt\},
    \nonumber
  }
  $F = \mD\mA\mA\T\mD+\frac{1-\gamma}{\alpha\beta}\mB^2$.
\end{theorem}

\begin{proof}
  Note that $\Null(B) = \Span(\1_n)$ is equivalent to $\Null(B^2) = \Span(\1_n)$ and recall that $D$ is double stochastic, we can verify that $F> 0$ by \cref{mA}, then we have
  \eqe{ \label{l_bound_3}
  \|\tw^k\|^2_{\mI-\alpha\beta F} \leq (1 - \alpha\beta\ue(F))\|\tw^k\|^2.
  }
  Also note that $\os(\mD) \leq 1$, according to \cref{min_max}, we have
  \eqe{
    \os(\mA\T\mD) \leq \os(\mA\T)\os(\mD) \leq \os(\mA).
    \nonumber
  }
  Let $\delta_2 = 1-\alpha\beta\ue(F)$, applying Weyl's inequality \cite{horn2012matrix} gives that
  \eqe{ \label{2002}
    \ue(F) =& \eta_1(F) \\
    \leq& \eta_1\lt(\frac{1-\gamma}{\alpha\beta}\mB^2\rt)+\ove(\mD\mA\mA\T\mD) \\
    =& \os^2(\mA\T\mD) \\
    \leq& \os^2(\mA),
  }
  where the first inequality holds since $\eta_1\lt(\frac{1-\gamma}{\alpha\beta}\mB^2\rt) = 0$.
  According to \cref{step_npga_g0_indicator} and $\mu \leq l$, we have
  \eqe{
    \alpha\beta <& \frac{1-\os(C)}{\os^2(\mA)},
    \nonumber
  }
  combining it with \cref{2002} gives that $\delta_2 \in (0, 1)$.
  Note that the condition of \cref{convergence_npga_g0_indicator} is sufficient for \cref{main_inequality}, then we have
  \eqe{
  &c_1\norm{\tx\+}^2 + c_2\norm{\tv\+}^2_{\mI-\gamma \mB^2} + c_3\norm{\ty\+}^2 \\
  \leq& \delta_1c_1\norm{\tx^k}^2 + c_2\norm{\tw^k}^2_{\mI-\mC-\alpha\beta \mA\mA\T} + \delta_3c_3\norm{\ty^k}^2 \\
  \leq& \delta_1c_1\norm{\tx^k}^2 + c_2\|\tv^k\|^2_{\mD(\mI-\mC-\alpha\beta \mA\mA\T)\mD} + \delta_3c_3\norm{\ty^k}^2 \\
  \leq& \delta_1c_1\norm{\tx^k}^2 + c_2\|\tv^k\|^2_{\mI-\mB^2-\alpha\beta \mD\mA\mA\T\mD} + \delta_3c_3\norm{\ty^k}^2  \\
  \leq& \delta_1c_1\norm{\tx^k}^2 + c_2\|\tv^k\|^2_{\mI-\alpha\beta F-\gamma\mB^2} + \delta_3c_3\norm{\ty^k}^2 \\
  \leq& \delta_1c_1\norm{\tx^k}^2 + \delta_2c_2\|\tv^k\|^2 - c_2\|\tv^k\|^2_{\gamma\mB^2} + \delta_3c_3\norm{\ty^k}^2 \\
  \leq& \delta_1c_1\norm{\tx^k}^2 + \delta_2c_2\|\tv^k\|^2_{\mI-\gamma\mB^2} + \delta_3c_3\norm{\ty^k}^2 \\
  \leq& \delta\lt(c_1\norm{\tx^k}^2 + c_2\|\tv^k\|^2_{\mI-\gamma\mB^2} + c_3\norm{\ty^k}^2\rt),
  \nonumber
  }
  where the second, the third, the fifth, and the sixth inequalities hold due to \cref{npga_error_g0_indicator}, \cref{DCB}, \cref{l_bound_3}, and $\delta_2 < 1$, respectively. It follows that
  \eqe{
    c_1\norm{\tx^k}^2 \leq \delta^k\lt(c_1\norm{\tx^0}^2 + c_2\|\tv^0\|^2_{\mI-\gamma\mB^2} + c_3\norm{\ty^0}^2\rt),
    \nonumber
  }
  which completes the proof.
\end{proof}

\begin{remark}
  As shown in \cref{convergence_npga_g0_indicator}, when \cref{indicator} holds, the linear convergence of NGPA can be guaranteed without assuming $\Null(C) = \Span(\1_n)$. The possibility of $C = 0$ introduces an interesting outcome: the upper bounds of the step-sizes of NPGA are independent of the network topology when $C = 0$, which expands their stability region, surpassing other versions of NPGA, including DCPA. Furthermore, the upper bound of $\gamma$ and the linear convergence rate are both improved compared to \cref{convergence_npga_g0}. Recall DCPA is a special case of NPGA, hence this improvement can also be applied to it. Also note that most of ATC versions of NPGA are with $C=0$, which are excluded in \cref{convergence_npga_g0}. However, as shown in in numerical experiments, ATC algorithms usually have much better performance than CTA ones. Therefore, the improvement of \cref{convergence_npga_g0_indicator} is highly significant and meaningful.
\end{remark}

\subsection{Case II: $h$ Is $l_h$-Smooth}
\begin{assumption} \label{h_smooth}
  $h$ is $l_h$-smooth for $l_h > 0$.
\end{assumption}

\cref{h_smooth} implies that $h^*$ is $\frac{1}{l_h}$-strongly convex, then we can obtain the linear convergence of NPGA without imposing the full rank condition on $A$, as shown in the following theorem.
\begin{theorem} \label{convergence_npga_h_smooth}
  Given \cref{condition_main_inequality,existence_saddle_point,G,convex,h_smooth}, let $\my^0=0$, $0 \leq \theta \leq 1$, and the positive step-sizes satisfy
  \eqe{ \label{step_npga_h_smooth}
    \alpha < \frac{1}{\max\lt\{\frac{2\theta l^2}{\mu}, 2l-\mu\rt\}}, \ \beta \leq \frac{\mu}{\os^2(\mA)\lt(\frac{1}{1-\os(C)}+\theta\rt)}, \\
    \gamma < \min\lt\{1, \frac{\lt(1+\frac{\beta}{nl_h}\rt)^2-1}{\lt(1+\frac{\beta}{nl_h}\rt)^2\os^2(B)}\rt\},
  }
  then there exists $\delta \in (0, 1)$ such that
  $$\norm{\mx^k-\mx^*}^2 = \mathcal{O}\lt(\delta^k\rt),$$
  where
  \eqe{
    \delta = \max\Bigg\{1-\alpha\lt(\mu-2\theta\alpha l^2\rt), \frac{1}{\lt(1+\frac{\beta}{nl_h}\rt)^2\lt(1-\gamma\os^2(B)\rt)}, \\1-\gamma\us^2(B)\Bigg\}.
    \nonumber
  }
\end{theorem}
\begin{proof}
  See Appendix \ref{appendix_convergence_npga_h_smooth}.
\end{proof}

\begin{figure*}[t]
  \begin{center}
    \subfigure[]{
      \begin{minipage}[b]{0.4\textwidth}
        \includegraphics[scale=0.35]{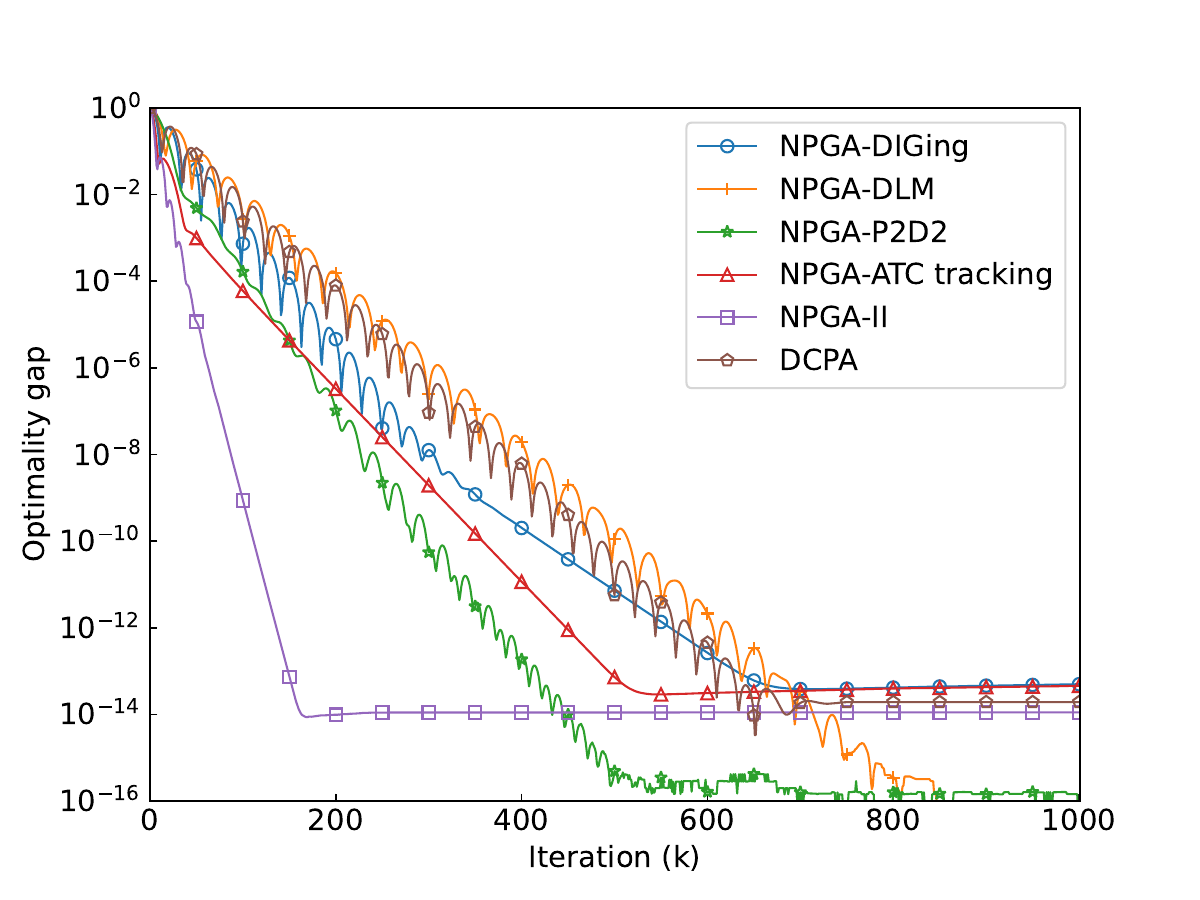}
      \end{minipage}
    }
    \subfigure[]{
      \begin{minipage}[b]{0.4\textwidth}
        \includegraphics[scale=0.35]{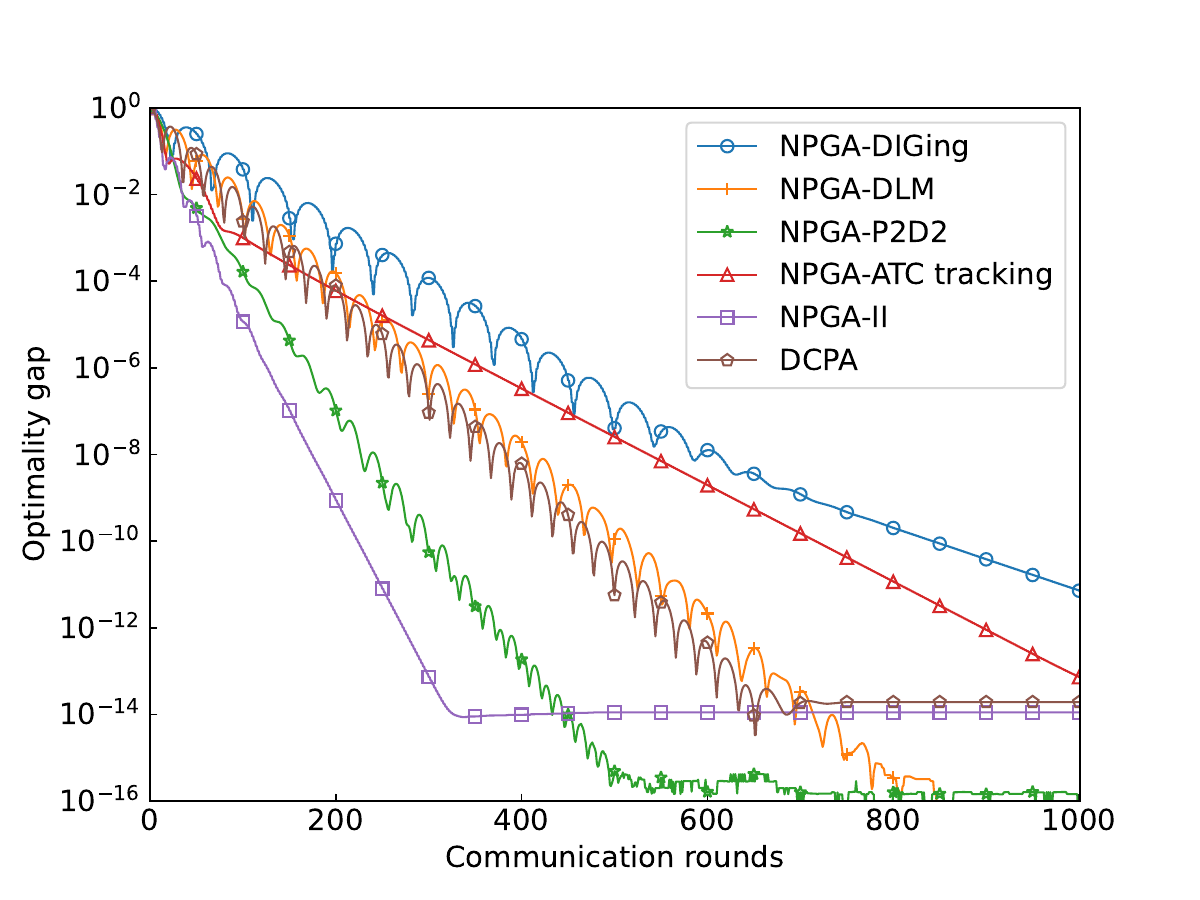}
      \end{minipage}
    }
    \caption {Result of Experiment I.}
    \label{fig1}
  \end{center}
\end{figure*}

\begin{figure*}[t]
  \begin{center}
    \subfigure[]{
      \begin{minipage}[b]{0.4\textwidth}
        \includegraphics[scale=0.35]{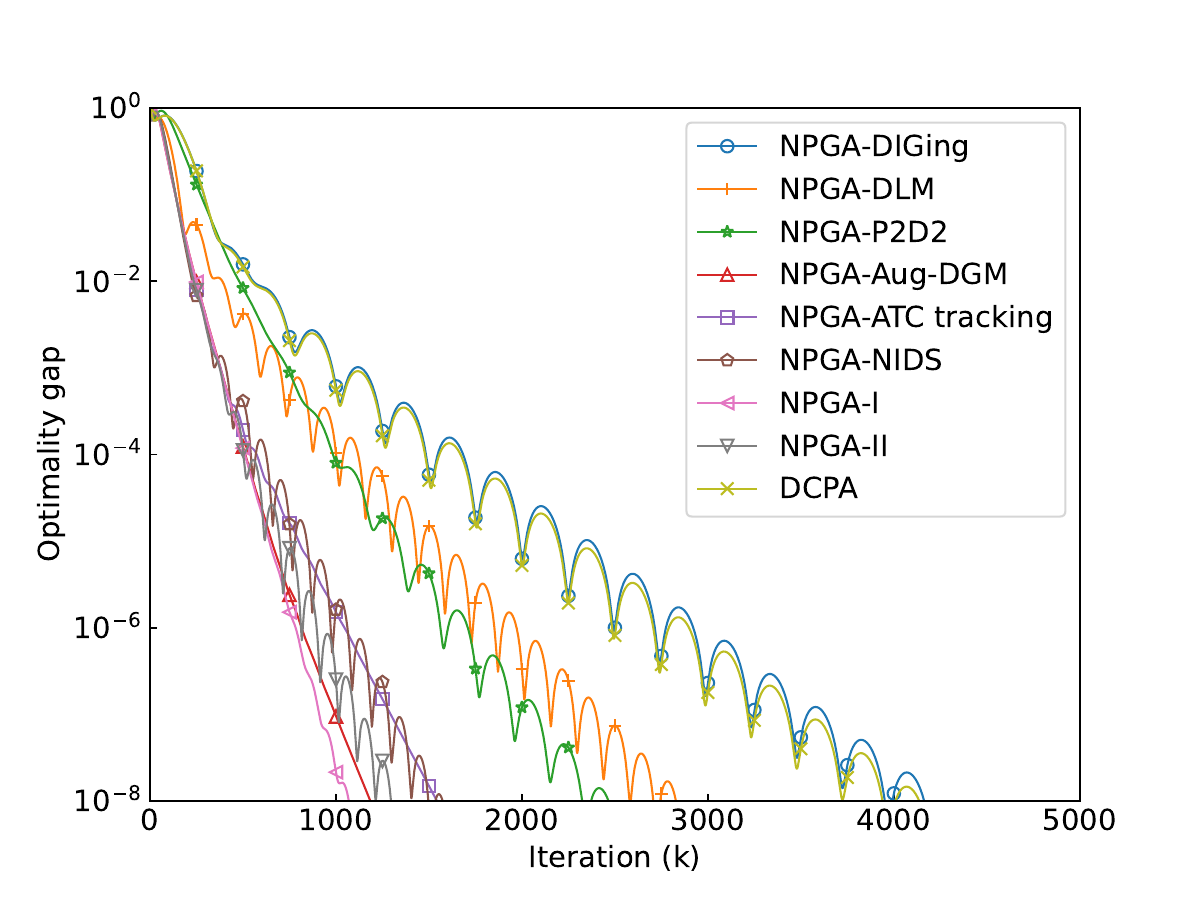}
      \end{minipage}
    }
    \subfigure[]{
      \begin{minipage}[b]{0.4\textwidth}
        \includegraphics[scale=0.35]{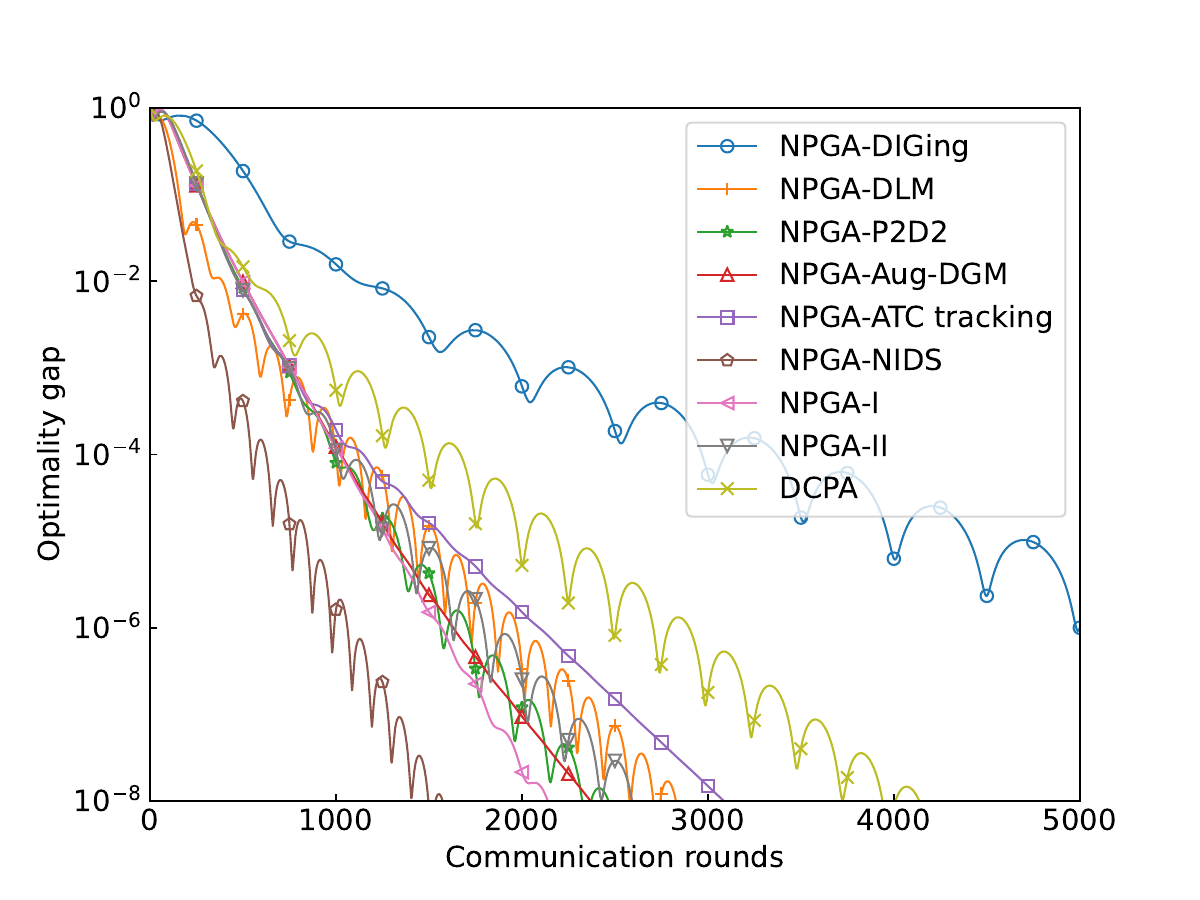}
      \end{minipage}
    }
    \caption {Result of Experiment II.}
    \label{fig2}
  \end{center}
\end{figure*}

\begin{figure*}[t]
  \begin{center}
    \subfigure[]{
      \begin{minipage}[b]{0.4\textwidth}
        \includegraphics[scale=0.35]{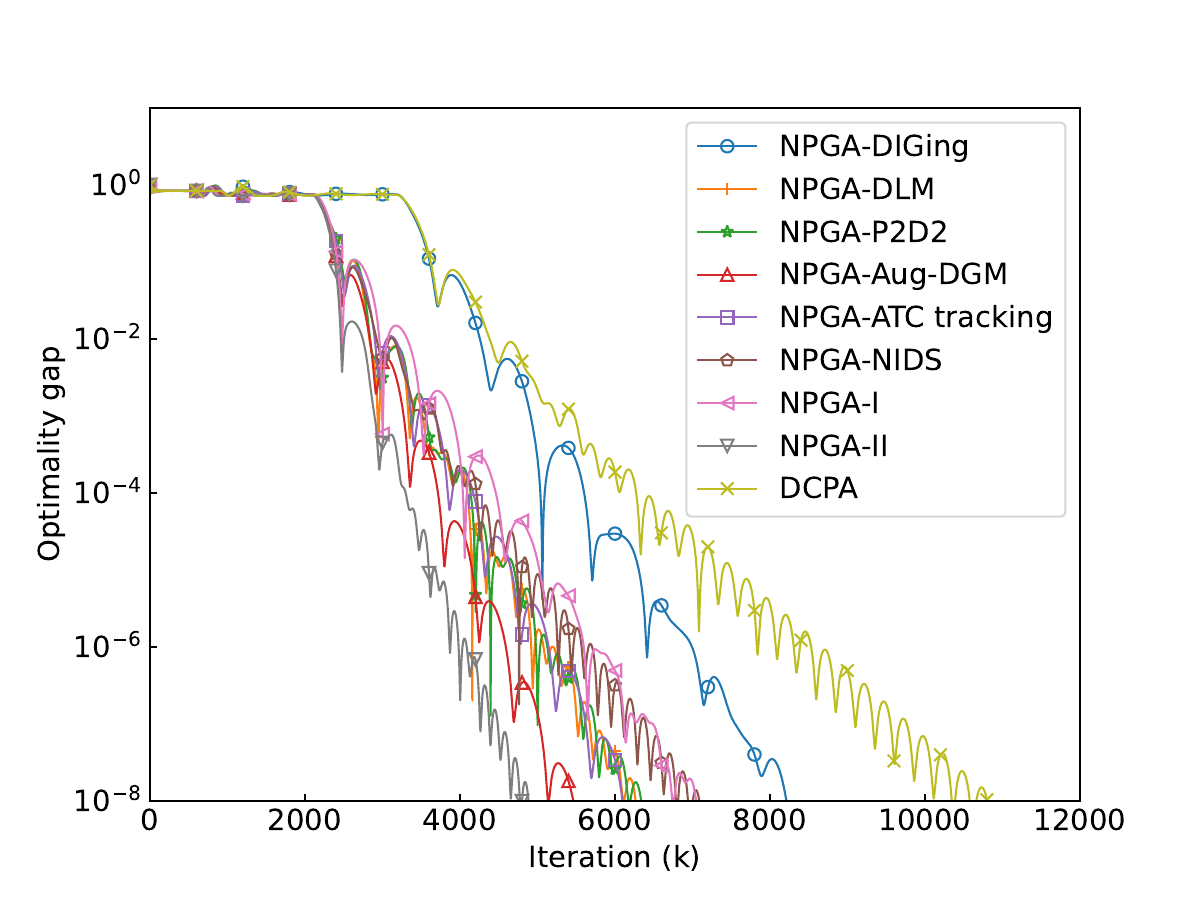}
      \end{minipage}
    }
    \subfigure[]{
      \begin{minipage}[b]{0.4\textwidth}
        \includegraphics[scale=0.35]{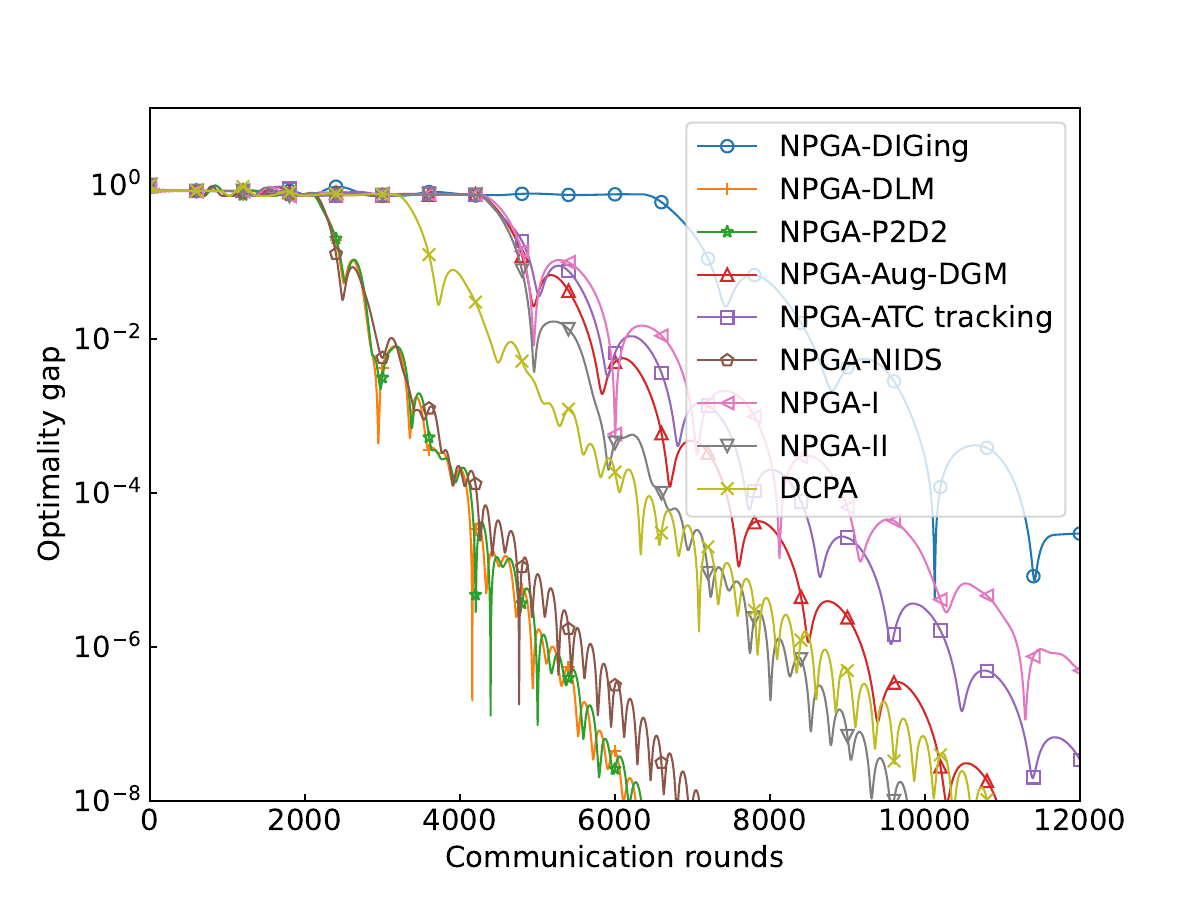}
      \end{minipage}
    }
    \caption {Result of Experiment III.}
    \label{fig3}
  \end{center}
\end{figure*}

\section{Numerical Experiments} \label{experiments}
\newcolumntype{Y}{>{\centering\arraybackslash}X}
\strutlongstacks{T}
In this section, we validate NPGA's convergence and evaluate the performance by solving vertical federated learning problems. Notice that NPGA-EXTRA and NPGA-Exact diffusion are the special cases of NPGA-P2D2 and NPGA-NIDS, respectively, hence we exclude them from the experiments. Though DCPA is also a special case of NPGA, we include it for the purpose of comparison.

Consider a scenario with $n$ parties,where each party has its own private data that cannot be accessed by others. The goal of federated learning systems is to train a learning model using all parties' data without compromising individual privacy. In vertical federated learning, the datasets among different parties consist of the same samples but with different features \cite{liu2022vertical}.

In particular, we focus on vertical federated learning for three types of regression problems: ridge regression, logistic regression, and elastic net regression, which correspond to the settings of \cref{convergence_npga_g0,convergence_npga_g0_DB}, \cref{convergence_npga_g0_indicator}, and\ \cref{convergence_npga_h_smooth}, respectively. For a given regression problem, let $X = [X_1, \cdots, X_n] \in \mR^{p \times d}$ and $Y \in \mR^{p}$ be the feature matrix and the label vector, respectively. Here, the local feature data $X_i \in \mR^{p \times d_i}$ belongs to the $i$-th party, and $\sum_{i=1}^n d_i = d$. Given this vertical federated learning setting, we can reformulate the original regression problem as \cref{pro2}, enabling us to solve it using NPGA.

\subsection{Experiment I: Ridge Regression}
One form of the ridge regression problem \cite{alghunaim2021dual} is formulated as
\eqe{ \label{ridge_regression}
  \min_{\theta \in \mR^d} \ \frac{1}{2}\norm{\theta}^2 + h(X\theta),
}
where $h$ is an indicator function given as
$$h(x)=\lt\{
  \begin{aligned}
    0,      \   & \norm{x - Y} \leq \delta, \\
    +\infty, \  & \norm{x - Y} > \delta.
  \end{aligned}\right.$$
We use $X\theta$ to denote the linear model since the last element of $\theta$ is the intercept. Therefore, $X$ is defined by $X = [X_{\text{raw}}, \1_n]$, where $X_{\text{raw}}$ is the raw feature matrix.
Considering of the vertical federated learning setting, \cref{ridge_regression} can be reformulated as
\eqe{ \label{decen_ridge_regression}
  \min_{\theta_i \in \mR^{d_i}} \ \frac{1}{2}\sum_{i=1}^n\norm{\theta_i}^2 + h\lt(\sum_{i=1}^nX_i\theta_i\rt),
}
which can be covered by \cref{pro2} with $f_i(x_i) = \frac{1}{2}\norm{x_i}^2$, $g_i=0$, and $A_i=X_i$. If $X$ has full row rank, \cref{decen_ridge_regression} will match the setting of \cref{convergence_npga_g0,convergence_npga_g0_DB}.

For this experiment, we use the Boston housing prices dataset \footnote{\href{http://lib.stat.cmu.edu/datasets/boston}{http://lib.stat.cmu.edu/datasets/boston}}, which have 13 features. We choose 10 samples to guarantee that $X \in \mR^{10 \times 14}$ has full row rank, and $X$ is divided into $[X_1, \cdots, X_{13}]$, where $X_i \in \mR^{10 \times 1}$ for $i \leq 12$, and $X_{13} \in \mR^{10 \times 2}$.

\subsection{Experiment II: Logistic Regression}
Given $X = [x_1, \cdots, x_p]\T \in \mR^{p \times d}$ and $Y = [y_1, \cdots, y_p]\T \in \{1,-1\}^p$, the logistic regression problem is defined as
\eqe{ \label{LR}
  \min_{\theta \in \mR^d} \frac{1}{p}\sum_{i=1}^p\ln\lt(1+\exp(-y_i\theta\T x_i)\rt) + \frac{\rho}{2}\norm{\theta}^2,
}
where $\rho>0$ is the regularization parameter. Considering of the vertical federated learning setting, let $X_i = [x_{i1}, \cdots, x_{ip}]\T$, \cref{LR} can be rewritten as
\eqe{ \label{LR2}
  \min_{\theta_i \in \mR^{d_i}} \frac{1}{p}\sum_{j=1}^p\ln\lt(1+\exp\lt(-y_j\sum_{i=1}^n\theta_i\T x_{ij}\rt)\rt) + \frac{\rho}{2}\sum_{i=1}^n\norm{\theta_i}^2.
}
By introducing a slack variable $z = [z_1, \cdots, z_p]\T = \sum_{i=1}^{n}X_i\theta_i$, we can reformulate \cref{LR2} as
\eqe{ \label{LR3}
  \min_{\theta_i \in \mR^{d_i}, z \in \mR^p} \ &\frac{1}{p}\sum_{j=1}^p\ln\lt(1+\exp\lt(-y_jz_j\rt)\rt) + \frac{\rho}{2}\sum_{i=1}^n\norm{\theta_i}^2 \\
  \text{s.t.} \ &\sum_{i=1}^nX_i\theta_i - z = 0,
}
which can be covered by \cref{pro2} with
\eqe{
&f_i(x_i) = \frac{\rho}{2}\|x_i\|^2, \ i = 1, \cdots, n, \\
&f_{n+1}(x_{n+1}) = \frac{1}{p}\sum_{j=1}^p\ln\lt(1+\exp\lt(-y_jx_{n+1,j}\rt)\rt), \\
&g_i=0, \ i = 1, \cdots, n+1, \\
&A_i = X_i, \ i = 1, \cdots, n, \ A_{n+1} = -\mI_p, \\
&h(x)=\lt\{
\begin{aligned}
  0,      \   & x = 0,    \\
  +\infty, \  & x \neq 0.
\end{aligned}\right.
\nonumber
}
A fact of \cref{LR3} is that $A = [X_1, \cdots, X_n, -\mI_p]$ always has full row rank, which enables it to match the setting of \cref{convergence_npga_g0_indicator} naturally.

For this experiment, we use the Covtype dataset \footnote{\href{https://www.csie.ntu.edu.tw/~cjlin/libsvmtools/datasets/binary.html}{https://www.csie.ntu.edu.tw/\~{}cjlin/libsvmtools/datasets/binary.html}}, which has 54 features. We choose 100 samples to constuct $X$ and divide it into $[X_1, \cdots, X_{27}]$, where $X_i \in \mR^{100 \times 2}$ for $i \leq 26$, and $X_{27} \in \mR^{100 \times 3}$.

\subsection{Experiment III: Elastic Net Regression}
The elastic net regression problem is formulated as
\eqe{ \label{elastic_net}
  \min_{\theta \in \mR^d} \ \frac{1}{2p}\norm{X\theta-Y}^2 + \alpha\rho\norm{\theta}_1 + \frac{\alpha(1-\rho)}{2}\norm{\theta}^2,
}
which is equivalent to
\eqe{ \label{elastic_net2}
  \min_{\theta \in \mR^d} \ h(X\theta) + \alpha\rho\norm{\theta}_1 + \frac{\alpha(1-\rho)}{2}\norm{\theta}^2,
}
where $h(x) = \frac{1}{2p}\norm{x-Y}^2$.
Considering of the vertical federated learning setting, \cref{elastic_net2} can be reformulated as
\eqe{ \label{decen_elastic_net}
  \min_{\theta_i \in \mR^{d_i}} \ \frac{\alpha(1-\rho)}{2}\sum_{i=1}^n\norm{\theta_i}^2 + \alpha\rho\sum_{i=1}^n\norm{\theta_i}_1 + h\lt(\sum_{i=1}^nX_i\theta_i\rt),
}
which can be covered by \cref{pro2} with $f_i(x_i) = \frac{\alpha(1-\rho)}{2}\norm{x_i}^2$, $g_i(x_i)=\alpha\rho\norm{x_i}_1$, and $A_i=X_i$. Note that $h$ is $\frac{1}{p}$-smooth, hence \cref{decen_elastic_net} matches the setting of \cref{convergence_npga_h_smooth}.

For this experiment, we still use the Boston housing prices dataset and choose 10 samples to construct $X$. The division mode of $X$ is also the same with Experiment I. The only difference is that we do not need to guarantee that $X$ has full row rank.

According to the problem formulations of the above three experiments, the numbers of parties are $13$, $28$, and $13$, respectively. We utilize the Erdos-Renyi model \cite{erdos1960evolution} with a connectivity probability of $0.3$ to generate the network topology among different parties, and the mixing matrix $W$ is constructed by the Laplacian method. When applying an algorithm to solve the problem, we try different combinations of step-sizes to obtain the fastest convergence rate.

The experimental results are shown in \cref{fig1,fig2,fig3}, where the optimality gap is defined as $\frac{\norm{\mx^k-\mx^*}}{\|\mx^0-\mx^*\|}$. Note that in various versions of NPGA, including DCPA, the computations of the gradient and the proximal operator represent the primary computational burden at each iteration. Furthermore, these computations are performed only once per iteration for both NPGA and DCPA. Consequently, the number of iterations can serve as a metric for comparing the computational costs between NPGA and DCPA. According to the experiment results, we can see that NPGA can achieve linear convergence for all problem settings, which support our theoretical results. Besides, certain versions of NPGA exhibit significantly faster convergence rates compared to DCPA in terms of both the numbers of iterations and communication rounds, which suggests that they outperforms DCPA in terms of both computational and communication efficiencies.

\section{Conclusion} \label{conclusion}
In this work, we study a class of decentralized constraint-coupled optimization problems and propose a novel nested primal-dual gradient algorithm called NPGA. NPGA not only serves as an algorithm but also offers an algorithmic framework, encompassing various
existing algorithms as special cases. We prove that NPGA can achieve linear convergence under the weakest known condition, and the theoretical convergence rate surpasses all known results. By designing different network matrices, we can derive various versions of NPGA and anlyze their convergences conveniently, providing an opportunity to design more efficient algorithms. Numerical experiments also confirm that NPGA exhibits a faster convergence rate compared to existing algorithms. Nevertheless, the linear convergence of NPGA is established under the assumption of a time-invariant network topology, which may not hold true in certain scenarios. Therefore, exploring the extension of NPGA's linear convergence to time-varying graphs will be a meaningful direction of future research.

\bibliographystyle{ieeetr}
\bibliography{D:/Research/Works/public/bib}

\begin{IEEEbiography}
  [{\includegraphics[width=1in,height=1.25in,clip,keepaspectratio]{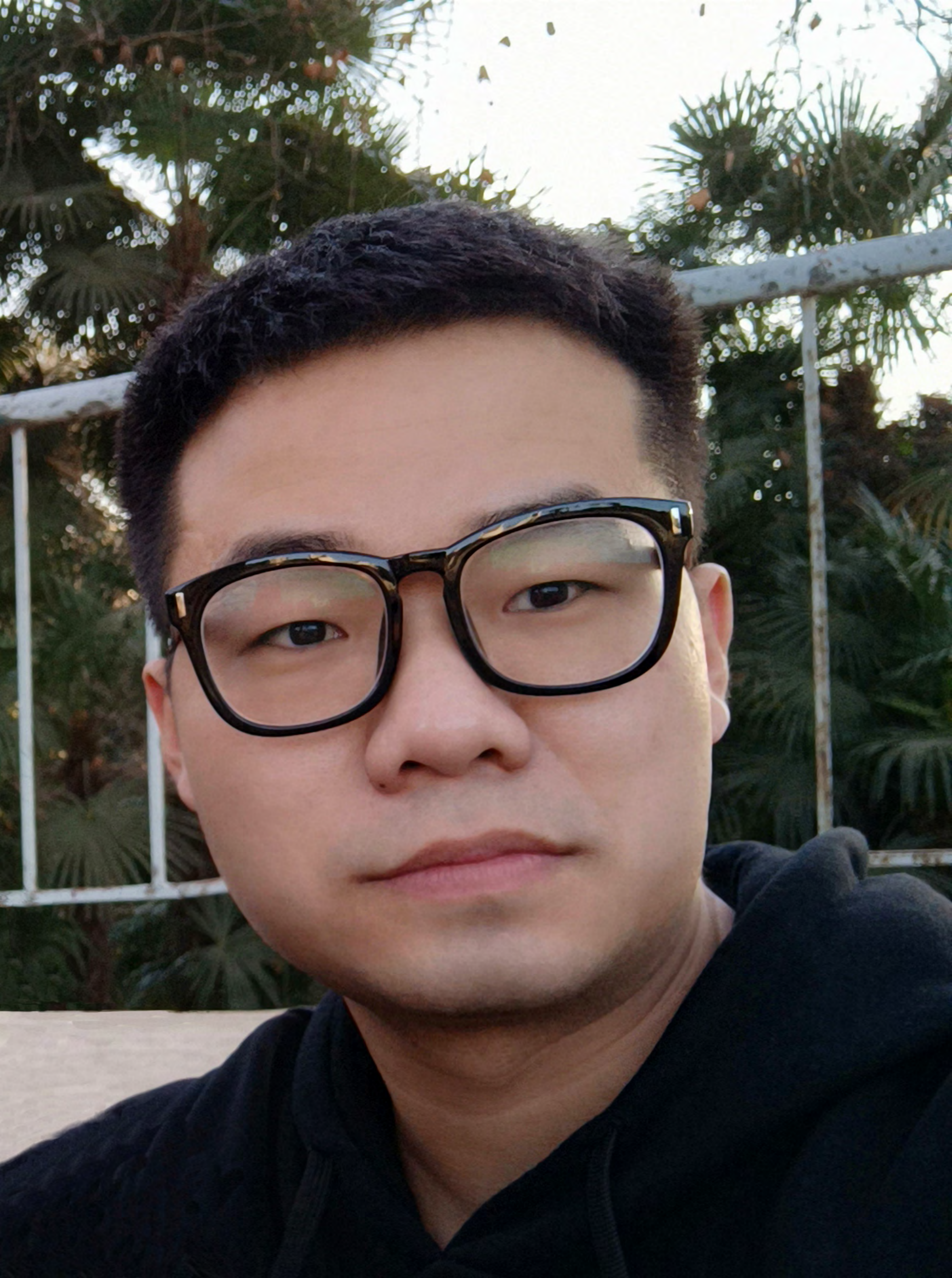}}]
  {Jingwang Li}
  received the B.S. degree in engineering management from Huazhong Agricultural University, Wuhan, China, in 2019 and the M.S. degree in control science and engineering from Huazhong University of Science and Technology, Wuhan, China, in 2022. His research interests include decentralized optimization and learning.
\end{IEEEbiography}

\begin{IEEEbiography}
  [{\includegraphics[width=1in,height=1.25in,clip,keepaspectratio]{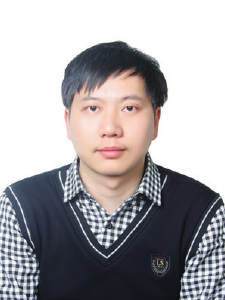}}]
  {Housheng Su}
  received his B.S. degree in automatic control and his M.S. degree in control theory and control engineering from Wuhan University of Technology, Wuhan, China, in 2002 and 2005, respectively, and his Ph.D. degree in control theory and control engineering from Shanghai Jiao Tong University, Shanghai, China, in 2008. From December 2008 to January 2010, he was a Postdoctoral researcher with the Department of Electronic Engineering, City University of Hong Kong, Hong Kong. Since November 2014, he has been a full professor with the School of Artificial Intelligence and Automation, Huazhong University of Science and Technology, Wuhan, China. He is an Associate Editor of IET Control Theory and Applications. His research interests lie in the areas of multi-agent coordination control theory and its applications to autonomous robotics and mobile sensor networks.
\end{IEEEbiography}

\appendices \label{appendix}
\section{} \label{appendix_min_max}
\begin{appendix_proof}[\cref{min_max}]
  Since $H$ and $HM\T MH$ are both symmetric, according to Courant-Fischer min-max theorem \cite{horn2012matrix}, we have
  \eqe{
    \lambda_i(H^2) =& \min_{\{S:\text{dim}(S)=i\}}\max_{\{x \in S, \|x\|=1\}}\|Hx\|^2, \\
    \lambda_i(HM\T MH) =& \min_{\{S:\text{dim}(S)=i\}}\max_{\{x \in S, \|x\|=1\}}\|MHx\|^2,
  }
  where $S$ is a subspace of $\mR^n$ and $\text{dim}(S)$ is the dimension of $S$. It follows that
  \eqe{
  \sigma_i(H) = \lambda_i(H^2)^{\frac{1}{2}} = \min_{\{S:\text{dim}(S)=i\}}\max_{\{x \in S, \|x\|=1\}}\|Hx\|
  }
  and
  \eqe{
  \sigma_i(MH) =& \lambda_i(HM\T MH)^{\frac{1}{2}} \\
  =& \min_{\{S:\text{dim}(S)=i\}}\max_{\{x \in S, \|x\|=1\}}\|MHx\| \\
  \leq& \os(M)\min_{\{S:\text{dim}(S)=i\}}\max_{\{x \in S, \|x\|=1\}}\|Hx\| \\
  =&\os(M)\sigma_i(H),
  }
  which completes the proof.
\end{appendix_proof}

\section{} \label{appendix_fixed_point}
\begin{appendix_proof}[\cref{fixed_point}]
  If $(\mx^*, \mv^*, \my^*, \ml^*)$ is a fixed point of NPGA, it holds that
  \begin{subequations} \label{npga_fixed_point}
    \begin{align}
      0     & \in \nabla f(\mx^*) + \mA\T\ml^* + \partial g(\mx^*), \label{npga_fixed_point_a} \\
      \mv^* & = \ml^* - \mC\ml^* - \mB\my^* + \beta\mA\mx^*, \label{npga_fixed_point_b}        \\
      0     & = \mB\mv^*, \label{npga_fixed_point_c}                                           \\
      \ml^* & = \prox_{\beta\mh}\lt(\mD\mv^*\rt). \label{npga_fixed_point_d}
    \end{align}
  \end{subequations}

  \textbf{Necessity}: According to \cref{condition_main_inequality,npga_fixed_point_c}, we know that $\mv^* = \1_n \otimes v^*$, where $v^* \in \mR^p$. Note that $D$ is a doubly stochastic matrix, thus $\ml^* = \mD\mv^* = \mv^*$. Combining it with \cref{npga_fixed_point_d}, we can further obtain that $\ml^* = \1_n \otimes \lambda^*$, where $\lambda^* \in \mR^p$. Then we can obtain \cref{opt_condition_a} from \cref{npga_fixed_point_a}. We also have
  \eqe{
    \lambda^* = \prox_{\frac{\beta}{n}h^*}(v^*),
    \nonumber
  }
  which is equivalent to
  \eqe{ \label{0029}
    \frac{n}{\beta}(v^* - \lambda^*) \in \partial h^*(\lambda^*).
  }
  Left multiplying \cref{npga_fixed_point_b} with $\1_n\T \otimes \mI_p$ gives that
  \eqe{
    nv^* &= n\lambda^* + \beta A\mx^*,
    \nonumber
  }
  applying it to \cref{0029} gives \cref{opt_condition_b}.

  \textbf{Sufficiency}: Let $\ml^* = \1_n \otimes \lambda^*$, then \cref{npga_fixed_point_a} holds. Let
  \eqe{ \label{0046}
    v^* = \lambda^* + \frac{\beta}{n} A\mx^*
  }
  and $\mv^* = \1_n \otimes v^*$, clearly \cref{npga_fixed_point_c} holds. Combining \cref{0046,opt_condition_b}, we can also obtain \cref{npga_fixed_point_d}. Note that $\mC\ml^* = 0$ and
  \eqe{
    \lt(\1_n\T \otimes \mI_p\rt)\lt(\ml^*-\mv^* + \beta\mA\mx^*\rt) = n(\lambda^*-v^*) + \beta A\mx^* = 0,
    \nonumber
  }
  which implies that
  $$\ml^*-\mv^* + \beta\mA\mx^* \in \mathbf{Null}\lt(\1_n\T \otimes \mI_p\rt) = \col(\mB).$$
  As a result, there must exist $\my^* \in \mR^{np}$ such that \cref{npga_fixed_point_b} holds, which completes the proof.
\end{appendix_proof}

\section{} \label{appendix_mA}
\begin{appendix_proof}[\cref{mA}]
  Obviously $\mathbf{M}\mA\mA\T\mathbf{M}+c\mathbf{H}$ is positive semi-definite. Assume that $\mx = [x_1\T, \cdots, x_n\T]\T \in \mR^{np}$ satisfies $\mx\T\mathbf{M}\mA\mA\T\mathbf{M}+c\mathbf{H}\mx = 0$, then we have $\mx\T\mathbf{M}\mA\mA\T\mathbf{M}\mx = 0$ and $\mx\T\mathbf{H}\mx = 0$. The latter implies that $x = x_1 = \cdots = x_n$, it follows that
  \eqe{
    0 =& \mx\T\mathbf{M}\mA\mA\T\mathbf{M}\mx \\
    =& \mx\T\mA\mA\T\mx \\
    =& x\T AA\T x,
  }
  where the first equality holds since $M$ is double stochastic. Note that \cref{A} implies that $AA\T > 0$, hence $x = 0$, which completes the proof.
\end{appendix_proof}

\section{} \label{appendix_main_inequality}
\begin{appendix_proof}[\cref{main_inequality}]
  According to \cref{npga_error_g0}, we have
  \eqe{ \label{32}
  &\norm{\tv\+}^2 \\
  =& \norm{\tw^k-\mC\tw^k+\beta \mA\lt(\tx\+ + \theta(\tx\+-\tx^k)\rt)}^2 \\
  &- 2\dotprod{\mB\ty^k, \tw^k-\mC\tw^k+\beta \mA\lt(\tx\+ + \theta(\tx\+-\tx^k)\rt)} \\
  &+ \norm{\ty^k}^2_{\mB^2}, \\
  &\norm{\ty\+}^2 \\
  =& \norm{\ty^k}^2 + \gamma^2\norm{\tv\+}^2_{\mB^2} + 2\gamma\dotprod{\mB\ty^k, \tv\+} \\
  =& \norm{\ty^k}^2 + \gamma^2\norm{\tv\+}^2_{\mB^2} - 2\gamma\norm{\ty^k}^2_{\mB^2} \\
  &+ 2\gamma\dotprod{\mB\ty^k, \tw^k-\mC\tw^k+\beta \mA\lt(\tx\+ + \theta(\tx\+-\tx^k)\rt)}.
  }
  Note that
  \eqe{
    &\dotprod{\tw^k-\mC\tw^k+\beta \mA\tx\+, \mA\lt(\tx\+-\tx^k\rt)} \\
    =& \dotprod{\tw^k-\mC\tw^k, \mA\lt(\tx\+-\tx^k\rt)} + \beta\dotprod{\mA\tx\+, \mA\lt(\tx\+-\tx^k\rt)} \\
    =& \dotprod{\tw^k-\mC\tw^k, \mA\lt(\tx\+-\tx^k\rt)} \\
    &+ \frac{\beta}{2}\lt(\norm{\mA\tx\+}^2 + \norm{\mA\lt(\tx\+-\tx^k\rt)}^2 - \norm{\mA\tx^2}\rt),
  }
  then we have
  \eqe{ \label{1132}
    &\norm{\tw^k-\mC\tw^k+\beta \mA\lt(\tx\+ + \theta(\tx\+-\tx^k)\rt)}^2 \\
    =& \norm{\tw^k-\mC\tw^k+\beta \mA\tx\+}^2 + \theta^2\beta^2\norm{\mA\lt(\tx\+-\tx^k\rt)}^2 \\
    &+ 2\theta\beta\dotprod{\tw^k-\mC\tw^k+\beta \mA\tx\+, \mA\lt(\tx\+-\tx^k\rt)} \\
    =& \norm{\tw^k-\mC\tw^k+\beta \mA\tx\+}^2 + (1+\theta)\theta\beta^2\norm{\mA\lt(\tx\+-\tx^k\rt)}^2 \\
    &+ \theta\beta^2\norm{\mA\tx\+}^2 - \theta\beta^2\norm{\mA\tx}^2 \\
    &+ 2\theta\beta\dotprod{\tw^k-\mC\tw^k, \mA\lt(\tx\+-\tx^k\rt)}.
  }
  In light of \cref{condition_main_inequality} and $\gamma < 1$, we know that $\mI-\gamma \mB^2 > 0$. Combining \cref{32,1132} gives that
  \eqe{ \label{1428}
  &\norm{\tv\+}^2_{\mI-\gamma \mB^2} + \frac{1}{\gamma}\norm{\ty\+}^2 \\
  =& \norm{\tw^k-\mC\tw^k+\beta \mA\tx\+}^2 + \frac{1}{\gamma}\norm{\ty^k}^2_{\mI-\gamma \mB^2} \\
  &+ (1+\theta)\theta\beta^2\norm{\mA\lt(\tx\+-\tx^k\rt)}^2  + \theta\beta^2\norm{\mA\tx\+}^2 \\
  &- \theta\beta^2\norm{\mA\tx}^2 + 2\theta\beta\dotprod{\tw^k-\mC\tw^k, \mA\lt(\tx\+-\tx^k\rt)}.
  }
  Note that
  \eqe{ \label{0130}
    &2\dotprod{\alpha\mA\T\lt(\tw^k-\mC\tw^k\rt), \tx\+-\tx^k} \\
    =& \norm{\alpha\mA\T\lt(\tw^k-\mC\tw^k\rt) + \tx\+-\tx^k}^2 \\
    &- \alpha^2\norm{\mA\T\lt(\tw^k-\mC\tw^k\rt)}^2 - \norm{\tx\+-\tx^k}^2 \\
    \leq& \alpha^2\norm{\mA\T\mC\tw^k + \lt(\nabla f(\mx^k)-\nabla f(\mx^*)\rt)}^2 - \norm{\tx\+-\tx^k}^2 \\
    \leq& 2\alpha^2\norm{\mA\T\mC\tw^k}^2 + 2\alpha^2\norm{\nabla f(\mx^k)-\nabla f(\mx^*)}^2 \\
    &- \norm{\tx\+-\tx^k}^2,
  }
  where the first inequality holds because of \cref{npga_error_g0}.
  Also note that
  \cref{condition_main_inequality} implies that $0 < \os(C) < 1$, applying Jensen's inequality gives that
  \eqe{ \label{jen1}
  \norm{\beta \mA\tx\+-\mC\tw^k}^2 &\leq \frac{\beta^2}{1-\os(C)}\norm{\mA\tx\+}^2 + \frac{1}{\os(C)}\norm{\tw^k}^2_{\mC^2} \\
  &\leq \frac{\beta^2\os^2(\mA)}{1-\os(C)}\norm{\tx\+}^2 + \norm{\tw^k}^2_{\mC},
  }
  it follows that
  \eqe{ \label{17}
  &\norm{\tw^k-\mC\tw^k+\beta \mA\tx\+}^2 \\
  =& \norm{\tw^k}^2 + \norm{\beta \mA\tx\+-\mC\tw^k}^2 - 2\norm{\tw^k}^2_\mC \\
  &+ 2\beta\dotprod{\mA\T\tw^k, \tx\+} \\
  =& \frac{\beta^2\os^2(\mA)}{1-\os(C)}\norm{\tx\+}^2 + \norm{\tw^k}^2_{\mI-\mC} - 2\alpha\beta\norm{\mA\T\tw^k}^2 \\
  &+ 2\beta\dotprod{\mA\T\tw^k, \tx^k-\alpha(\nabla f(\mx^k)-\nabla f(\mx^*))}.
  }
  Applying \cref{0130,17} to \cref{1428} gives that
  \eqe{ \label{38}
  &\norm{\tv\+}^2_{\mI-\gamma \mB^2} + \frac{1}{\gamma}\norm{\ty\+}^2 \\
  \leq& \beta^2\os^2(\mA)\lt(\frac{1}{1-\os(C)}+\theta\rt)\norm{\tx\+}^2 + \norm{\tw^k}^2_{\mI-\mC} \\
  & - 2\alpha\beta\norm{\mA\T\tw^k}^2 + \frac{1}{\gamma}\norm{\ty^k}^2_{\mI-\gamma \mB^2} \\
  &+ 2\beta\dotprod{\mA\T\tw^k, \tx^k-\alpha(\nabla f(\mx^k)-\nabla f(\mx^*))} \\
  &+ \theta(1+\theta)\beta^2\|\mA(\tx\+-\tx^k)\|^2 - \frac{\theta\beta}{\alpha}\norm{\tx\+-\tx^k}^2 \\
  &+ 2\theta\alpha\beta\norm{\mA\T\mC\tw^k}^2 + 2\theta\alpha\beta\norm{\nabla f(\mx^k)-\nabla f(\mx^*)}^2.
  }

  According to \cref{npga_error_g0}, we can also obtain that
  \eqe{
    \norm{\tx\+}^2 =& \norm{\tx^k-\alpha\lt(\nabla f(\mx^k)-\nabla f(\mx^*)\rt)}^2 + \alpha^2\norm{\mA\T\tw^k}^2 \\
    &- 2\alpha\dotprod{\mA\T\tw^k, \tx^k-\alpha\lt(\nabla f(\mx^k)-\nabla f(\mx^*)\rt)},
  }
  then we have
  \eqe{ \label{39}
  &\norm{\tx\+}^2 + c_2\lt(\norm{\tv\+}^2_{\mI-\gamma \mB^2} + \frac{1}{\gamma}\norm{\ty\+}^2\rt) \\
  \leq& \alpha\beta\os^2(\mA)\lt(\frac{1}{1-\os(C)}+\theta\rt)\norm{\tx\+}^2 \\
  &+ \norm{\tx^k-\alpha(\nabla f(\mx^k)-\nabla f(\mx^*))}^2 \\
  &+ c_2\norm{\tw^k}^2_{\mI-\mC} - \alpha^2\norm{\mA\T\tw^k}^2 + c_3\norm{\ty^k}^2_{\mI-\gamma \mB^2} \\
  &+ \theta(1+\theta)\alpha\beta\norm{\mA(\tx\+-\tx^k)}^2- \theta\norm{\tx\+-\tx^k}^2 \\
  & + 2\theta\alpha^2\norm{\nabla f(\mx^k)-\nabla f(\mx^*)}^2 + 2\theta\alpha^2\norm{\mA\T\mC\tw^k}^2.
  }
  According to \cref{step_npga_g0} and $\mu \leq l$, we have
  \eqe{ \label{21}
    \alpha\beta <& \frac{1}{(1+2\theta)\os^2(\mA)\lt(\frac{1}{1-\os(C)}+\theta\rt)},
  }
  thus $c_1>0$ and $\mI-\mC-\alpha\beta \mA\mA\T > 0$, it follows that
  \eqe{ \label{139}
  &c_1\norm{\tx\+}^2 + c_2\norm{\tv\+}^2_{\mI-\gamma \mB^2} + c_3\norm{\ty\+}^2 \\
  \leq& \norm{\tx^k-\alpha(\nabla f(\mx^k)-\nabla f(\mx^*))}^2 \\
  &+ c_2\norm{\tw^k}^2_{\mI-\mC-\alpha\beta \mA\mA\T} + c_3\norm{\ty^k}^2_{\mI-\gamma \mB^2} \\
  &+ \theta(1+\theta)\alpha\beta\norm{\mA(\tx\+-\tx^k)}^2- \theta\norm{\tx\+-\tx^k}^2 \\
  & + 2\theta\alpha^2\norm{\nabla f(\mx^k)-\nabla f(\mx^*)}^2 + 2\theta\alpha^2\norm{\mA\T\mC\tw^k}^2.
  }

  Recall \cref{convex} and $\alpha l(1+2\theta) < 1$, applying \cref{convex_inequality} gives that
  \eqe{ \label{strconvex}
    &\norm{|\tx^k-\alpha(\nabla f(\mx^k)-\nabla f(\mx^*))}^2 \\
    &+ 2\theta\alpha^2\norm{\nabla f(\mx^k)-\nabla f(\mx^*)}^2 \\
    =& \norm{\tx^k}^2 - 2\alpha \dotprod{\tx^k, \nabla f(\mx^k)-\nabla f(\mx^*)} \\
    &+ (1+2\theta)\alpha^2\norm{\nabla f(\mx^k)-\nabla f(\mx^*)}^2 \\
    \leq& \norm{\tx^k}^2 - \alpha(2-\alpha l(1+2\theta)) \dotprod{\tx^k, \nabla f(\mx^k)-\nabla f(\mx^*)} \\
    \leq& (1-\alpha\mu(2-\alpha l(1+2\theta)))\norm{\tx^k}^2.
  }
  Let $\delta_1 = 1-\alpha\mu(1-\alpha l(1+2\theta))$, according to \cref{step_npga_g0}, we can easily verify that $\delta_1 \in (0, 1)$ and $\mu \geq \beta\os^2(\mA)\lt(\frac{1}{1-\os(C)}+\theta\rt)$, then we have
  \eqe{ \label{x_bound}
    &(1-\alpha\mu(2-\alpha l(1+2\theta)))\norm{\tx^k}^2 \\
    =& \delta_1\norm{\tx^k}^2 - \alpha\mu\|\tx^k\|^2 \\
    =& \delta_1c_1\norm{\tx^k}^2 \\
    &- \alpha\lt(\mu - \delta_1\lt(\beta\os^2(\mA)\lt(\frac{1}{1-\os(C)}+\theta\rt)\rt)\rt)\norm{x^k}^2 \\
    \leq& \delta_1c_1\norm{\tx^k}^2.
  }
  Note that $\my^0=0$ implies that $\my^k\in \col(\mB)$, then $\ty^k\in \col(\mB)$, it follows that
  \eqe{ \label{y_bound}
  \norm{\ty^k}^2_{\mI-\gamma \mB^2} \leq (1-\gamma\us^2(B))\norm{\ty^k}^2.
  }
  Let $\delta_3 = 1-\gamma\us^2(B)$, obviously $\delta_3 \in (0, 1)$.
  An obvious result of \cref{21} is
  \eqe{
    \alpha\beta <& \frac{1}{\os^2(\mA)\lt(1+\theta\rt)},
  }
  then we have
  \eqe{ \label{2351}
    &\theta(1+\theta)\alpha\beta\|\mA(\tx\+-\tx^k)\|^2 - \theta\norm{\tx\+-\tx^k}^2 \\
    \leq& \theta ((1+\theta)\alpha\beta\os^2(\mA) - 1)\norm{\tx\+-\tx^k}^2 \leq 0.
  }
  Also note that
  \eqe{ \label{2352}
  \norm{\mA\T\mC\tw^k}^2 \leq \os(\mA)^2\norm{\tw^k}^2_{\mC^2} \leq \os(C)\os(\mA)^2\norm{\tw^k}^2_{\mC},
  }
  then the proof is completeed by applying \cref{x_bound,y_bound,2351,2352} to \cref{139}.
\end{appendix_proof}

\section{} \label{appendix_main_inequality_C}
\begin{appendix_proof}[\cref{main_inequality_C}]
  Given \cref{A}, $\Null(C) = \Span(\1_n)$, and $\gamma<1$, we can verify that $E> 0$ using \cref{mA}, then we have
  \eqe{ \label{l_bound}
  \|\tw^k\|^2_{\mI-\alpha\beta E} \leq (1 - \alpha\beta\ue(E))\|\tw^k\|^2.
  }
  Let $\delta_2 = 1 - \alpha\beta\ue(E)$, in light of Weyl's inequality \cite{horn2012matrix} and the positive definiteness of $E$, we have
  $$\ue(E)=\eta_1(E)\leq \eta_1\lt(\frac{1}{2\alpha\beta}\mC\rt)+\ove(\mA\mA\T)=\os^2(\mA),$$
  then we can verify that $\delta_2 \in (0, 1)$ by combining it with \cref{21}.

  Note that
  \eqe{
    2\theta\alpha\beta\os(C)\os^2(\mA) < \frac{2\theta\os(C)}{(1+2\theta)\lt(\frac{1}{1-\os(C)}+\theta\rt)} \leq \frac{2\theta}{(1+2\theta)\lt(1+\theta\rt)},
    \nonumber
  }
  define $\Phi(\theta) = \frac{2\theta}{(1+2\theta)\lt(1+\theta\rt)}$, by studying its derivative, we can obtain that
  \eqe{ \label{0023}
    \Phi(\theta) \leq \Phi\lt(\frac{1}{\sqrt{2}}\rt) < \frac{1}{2}, \ \forall \theta \geq 0.
  }
  Applying \cref{l_bound,0023} to \cref{main_inequality} gives that
  \eqe{ \label{391}
  &c_1\norm{\tx\+}^2 + c_2\norm{\tv\+}^2_{\mI-\gamma \mB^2} + c_3\norm{\ty\+}^2 \\
  \leq& \delta_1c_1\norm{\tx^k}^2 + \delta_2c_2\norm{\tw^k}^2 + \delta_3c_3\norm{\ty^k}^2 \\
  &- \lt(\frac{1}{2} - 2\theta\alpha\beta\os(C)\os^2(\mA)\rt)c_2\norm{\tw^k}^2_{\mC} \\
  \leq& \delta_1c_1\norm{\tx^k}^2 + \delta_2c_2\norm{\tw^k}^2 + \delta_3c_3\norm{\ty^k}^2,
  }
  which completes the proof.
\end{appendix_proof}

\section{} \label{appendix_convergence_npga_h_smooth}
\begin{appendix_proof}[\cref{convergence_npga_h_smooth}]
  Using the same analysis steps with \cref{main_inequality}, we can obtain that
  \eqe{ \label{2128}
  &\norm{\tv\+}^2_{\mI-\gamma \mB^2} + \frac{1}{\gamma}\norm{\ty\+}^2 \\
  =& \norm{\tw^k-\mC\tw^k+\beta \mA\tx\+}^2 + \frac{1}{\gamma}\norm{\ty^k}^2_{\mI-\gamma \mB^2} \\
  &+ (1+\theta)\theta\beta^2\norm{\mA\lt(\tx\+-\tx^k\rt)}^2  + \theta\beta^2\norm{\mA\tx\+}^2 \\
  &- \theta\beta^2\norm{\mA\tx}^2 + 2\theta\beta\dotprod{\tw^k-\mC\tw^k, \mA\lt(\tx\+-\tx^k\rt)}.
  }
  Recall \cref{jen1}, we have
  \eqe{
  &\norm{\tw^k-\mC\tw^k+\beta \mA\tx\+}^2 \\
  =& \norm{\tw^k}^2 + \|\beta \mA\tx\+-\mC\tw^k\|^2 - 2\norm{\tw^k}^2_\mC \\
  &+ 2\beta\dotprod{\mA\T\tw^k, \tx\+} \\
  =& \frac{\beta^2\os^2(\mA)}{1-\os(C)}\norm{\tx\+}^2 + \norm{\tw^k}^2_{\mI-\mC} + 2\beta\dotprod{\mA\T\tw^k, \tx\+},
  \nonumber
  }
  applying it and \cref{0130} to \cref{2128} gives that
  \eqe{ \label{2220}
  &\norm{\tv\+}^2_{\mI-\gamma \mB^2} + \frac{1}{\gamma}\norm{\ty\+}^2 \\
  \leq& \beta^2\os^2(\mA)\lt(\frac{1}{1-\os(C)}+\theta\rt)\norm{\tx\+}^2 + \norm{\tw^k}^2_{\mI-\mC} \\
  &+ 2\beta\dotprod{\mA\T\tw^k, \tx\+} + \frac{1}{\gamma}\norm{\ty^k}^2_{\mI-\gamma \mB^2} \\
  &+ \theta(1+\theta)\beta^2\|\mA(\tx\+-\tx^k)\|^2- \frac{\theta\beta}{\alpha}\norm{\tx\+-\tx^k}^2 \\
  &+ 2\theta\alpha\beta\norm{\mA\T\mC\tw^k}^2 + 2\theta\alpha\beta\norm{\nabla f(\mx^k)-\nabla f(\mx^*)}^2.
  }

  Since $g$ is convex, we have
  \eqe{
    0 \leq& 2\alpha\dotprod{\tx\+, \mathcal{S}_g(\mx\+)-\mathcal{S}_g(\mx^*)} \\
    =& 2\dotprod{\tx\+, \tx^k - \tx\+} \\
    &- 2\dotprod{\tx\+, \alpha\lt(\nabla f(\mx^k)-\nabla f(\mx^*)\rt) + \alpha\mA\T\tw^k},
    \nonumber
  }
  note that
  \eqe{
    2\dotprod{\tx\+, \tx^k - \tx\+} = \norm{\tx^k}^2 - \norm{\tx\+}^2 - \norm{\tx\+-\tx^k}^2,
    \nonumber
  }
  then we have
  \eqe{ \label{2218}
    \norm{\tx\+}^2 \leq& \norm{\tx^k}^2 - \norm{\tx\+-\tx^k}^2 \\
    &- 2\alpha\dotprod{\tx\+, \nabla f(\mx^k)-\nabla f(\mx^*)} \\
    &- 2\alpha\dotprod{\tx\+, \mA\T\tw^k}.
  }
  Let $c_2 = \frac{\alpha}{\beta}$ and $c_3 = \frac{\alpha}{\beta\gamma}$, combining \cref{2218,2220} gives that
  \eqe{ \label{1239}
  &\norm{\tx\+}^2 + c_2\lt(\norm{\tv\+}^2_{\mI-\gamma \mB^2} + \frac{1}{\gamma}\norm{\ty\+}^2\rt) \\
  \leq& \alpha\beta\os^2(\mA)\lt(\frac{1}{1-\os(C)}+\theta\rt)\norm{\tx\+}^2 \\
  &+ \norm{\tx^k}^2 - 2\alpha\dotprod{\tx\+, \nabla f(\mx^k)-\nabla f(\mx^*)} \\
  &+ c_2\norm{\tw^k}^2_{\mI-\mC} + c_3\norm{\ty^k}^2_{\mI-\gamma \mB^2} - \norm{\tx\+-\tx^k}^2 \\
  &+ \theta(1+\theta)\alpha\beta\|\mA(\tx\+-\tx^k)\|^2 - \theta\norm{\tx\+-\tx^k}^2 \\
  &+ 2\theta\alpha^2\norm{\nabla f(\mx^k)-\nabla f(\mx^*)}^2 + 2\theta\alpha^2\norm{\mA\T\mC\tw^k}^2.
  }

  Recall that $f$ is $\mu$-strongly convex and $l$-smooth, we have
  \eqe{ \label{0146}
    \norm{\nabla f(\mx^k)-\nabla f(\mx^*)}^2 \leq l^2\norm{\tx^k}
  }
  and
  \eqe{ \label{0147}
    &\dotprod{\tx\+, \nabla f(\mx^k)-\nabla f(\mx^*)} \\
    =& \dotprod{\mx\+ - \mx^*, \nabla f(\mx^k)} - \dotprod{\mx\+ - \mx^*, \nabla f(\mx^*)} \\
    =& \dotprod{\mx\+ - \mx^k, \nabla f(\mx^k)} + \dotprod{\mx^k - \mx^*, \nabla f(\mx^k)} \\
    &- \dotprod{\mx\+ - \mx^*, \nabla f(\mx^*)} \\
    =& -\dotprod{\mx\+ - \mx^k, \nabla f(\mx\+) - \nabla f(\mx^k)} \\
    &- \dotprod{\mx^k - \mx\+, \nabla f(\mx\+)} - \dotprod{\mx^* - \mx^k, \nabla f(\mx^k)} \\
    &- \dotprod{\mx\+ - \mx^*, \nabla f(\mx^*)} \\
    \geq& -l\norm{\mx\+ - \mx^k}^2 \\
    &- \lt(f(\mx^k) - f(\mx\+) - \frac{\mu}{2}\norm{\mx\+-\mx^k}^2\rt) \\
    &- \lt(f(\mx^*) - f(\mx^k) - \frac{\mu}{2}\norm{\tx^k}^2\rt) \\
    &- \lt(f(\mx\+) - f(\mx^*) - \frac{\mu}{2}\norm{\tx\+}^2\rt) \\
    \geq& -\lt(l-\frac{\mu}{2}\rt)\norm{\tx\+ - \tx^k}^2 + \frac{\mu}{2}\lt(\norm{\tx\+}^2 + \norm{\tx^k}^2\rt) \\
  }
  where the first inequality holds due to \cref{convex_inequality}.
  Applying \cref{2352,0146,0147} to \cref{1239} gives that
  \eqe{ \label{2239}
  &\norm{\tx\+}^2 + c_2\norm{\tv\+}^2_{\mI-\gamma \mB^2} + c_3\norm{\ty\+}^2 \\
  \leq& -\alpha\lt(\mu-\beta\os^2(\mA)\lt(\frac{1}{1-\os(C)}+\theta\rt)\rt)\norm{\tx\+}^2 \\
  &+ (1-\alpha\lt(\mu-2\theta\alpha l^2\rt))\norm{\tx^k}^2 + c_2\norm{\tw^k}^2 \\
  &+ (1-\gamma\us^2(B))c_3\norm{\ty^k}^2 \\
  &- (1-\alpha(2l-\mu))\norm{\tx\+-\tx^k}^2 \\
  &- \theta (1-(1+\theta)\alpha\beta\os^2(\mA))\norm{\tx\+-\tx^k}^2 \\
  &- \lt(1 - 2\theta\alpha\beta\os^2(\mA)\rt)c_2\norm{\tw^k}^2_{\mC}.
  }
  According to \cref{step_npga_h_smooth} and $0 \leq \theta \leq 1$, we have
  \eqe{ \label{121}
    \mu &\geq \beta\os^2(\mA)\lt(\frac{1}{1-\os(C)}+\theta\rt), \ 1 > \alpha(2l-\mu), \\
    \alpha\beta &< \frac{1}{\os^2(\mA)\lt(1+\theta\rt)}, \ 2\theta\alpha\beta\os^2(\mA) < \frac{2\theta}{1+\theta} \leq 1,
  }
  it follows that
  \eqe{ \label{1139}
  &\norm{\tx\+}^2 + c_2\norm{\tv\+}^2_{\mI-\gamma \mB^2} + c_3\norm{\ty\+}^2 \\
  \leq& \delta_1\norm{\tx^k}^2 + c_2\norm{\tw^k}^2 + \delta_3 c_3\norm{\ty^k}^2,
  }
  where $\delta_1 = 1-\alpha\lt(\mu-2\theta\alpha l^2\rt)$ and $\delta_3 = 1-\gamma\us^2(B)$.

  Note that \cref{h_smooth} implies that $\mh$ is $\frac{1}{nl_h}$-strongly convex. According to \cref{npga_l,npga_fixed_point_d}, we have
  \eqe{
    \mD\tv\+ = \tw\+ + \beta \lt(\mathcal{S}_{\mh}\lt(\ml\+\rt) - \mathcal{S}_{\mh}\lt(\ml^*\rt)\rt),
    \nonumber
  }
  it follows that
  \eqe{
    &\dotprod{\tw\+, \mD\tv\+} \\
    =& \norm{\tw\+}^2 + \beta \dotprod{\tw\+, \mathcal{S}_{\mh}\lt(\ml\+\rt) - \mathcal{S}_{\mh}\lt(\ml^*\rt)} \\
    \geq& \lt(1+\frac{\beta}{nl_h}\rt)\norm{\tw\+}^2,
    \nonumber
  }
  where the inequality holds due to the strong convexity of $\mh$. Applying the Cauchy-Schwarz inequality gives that
  \eqe{
    \lt(1+\frac{\beta}{nl_h}\rt)\norm{\tw\+}^2 \leq \norm{\tw\+}\norm{\mD\tv\+} \leq \norm{\tw\+}\norm{\tv\+},
    \nonumber
  }
  it follows that
  \eqe{
    \lt(1+\frac{\beta}{nl_h}\rt)\norm{\tw\+} \leq \norm{\tv\+}.
    \nonumber
  }
  Also note that
  \eqe{
  \norm{\tv\+}^2_{\mI-\gamma \mB^2} \geq \lt(1-\gamma\os^2(B)\rt)\norm{\tv\+}^2,
  \nonumber
  }
  then we have
  \eqe{ \label{l_bound_2}
  \norm{\tv\+}^2_{\mI-\gamma \mB^2} \geq \lt(1-\gamma\os^2(B)\rt)\lt(1+\frac{\beta}{nl_h}\rt)^2\|\tw\+\|^2.
  }
  Applying \cref{l_bound_2} to \cref{1139} gives that
  \eqe{
    &\norm{\tx\+}^2 + \lt(1+\frac{\beta}{nl_h}\rt)^2\lt(1-\gamma\os^2(B)\rt)c_2\norm{\tw\+}^2 \\
    &+ c_3\norm{\ty\+}^2 \\
    \leq& \delta_1\norm{\tx^k}^2 + c_2\norm{\tw^k}^2 + \delta_3c_3\norm{\ty^k}^2.
    \nonumber
  }
  Let $\delta_2 = \frac{1}{\lt(1+\frac{\beta}{nl_h}\rt)^2\lt(1-\gamma\os^2(B)\rt)}$, then we have
  \eqe{
    &\delta_2\norm{\tx\+}^2 + c_2\norm{\tw\+}^2 + c_3\delta_2\norm{\ty\+}^2 \\
    \leq& \delta_1\delta_2\|\tx^k\|^2 + \delta_2c_2\norm{\tw^k}^2 + \delta_3c_3\delta_2\norm{\ty^k}^2 \\
    \leq& \delta\lt(\delta_2\|\tx^k\|^2 + c_2\norm{\tw^k}^2 + c_3\delta_2\norm{\ty^k}^2\rt),
    \nonumber
  }
  it follows that
  \eqe{
    \delta_2\norm{\tx^k}^2 \leq \delta^k\lt(\delta_2\norm{\tx^0}^2 + c_2\norm{\tw^0}^2 + c_3\delta_2\norm{\ty^0}^2\rt).
    \nonumber
  }
  According to \cref{condition_main_inequality,step_npga_h_smooth}, we can easily verify that $\delta_1, \delta_2, \delta_3 \in (0, 1)$, thus $\delta \in (0, 1)$, which completes the proof.
\end{appendix_proof}

\end{document}